\documentclass[12pt, a4paper]{amsart}

\usepackage{amsmath,amssymb}

\usepackage[colorlinks=true, pdfstartview=FitV, linkcolor=blue, citecolor=blue, urlcolor=blue]{hyperref}
\usepackage{color}
\usepackage{enumitem}
\usepackage{setspace}

\newtheorem{theorem}{Theorem}[section]
\newtheorem{lemma}[theorem]{Lemma}
\newtheorem{proposition}[theorem]{Proposition}

\theoremstyle{definition}

\theoremstyle{remark}
\newtheorem{remark}[theorem]{Remark}
\numberwithin{equation}{section}

\newcommand{\R}{\mathbb{R}}

\newcommand{\rn}{\mathbb{R}^n}

\begin{document}

\title[ ]
{Muckenhoupt-type conditions on weighted Morrey spaces}

\author{Javier Duoandikoetxea and Marcel Rosenthal} 
\address{Universidad del Pa\'is Vasco/Euskal Herriko Unibertsitatea, Departamento de Matem\'a\-ti\-cas, 
Apartado 644, 48080 Bilbao, Spain}

\email{javier.duoandikoetxea@ehu.eus}

\address{Stuttgart, Germany} \email{marcel.rosenthal@uni-jena.de} 

  %, 

\begin{abstract}
We define a Muckenhoup-type condition on weighted Morrey spaces using the K\"othe dual of the space. We show that the condition is necessary and sufficient for the boundedness of the maximal operator defined with balls centered at the origin on weighted Morrey spaces. A modified condition characterizes the weighted inequalities for the Calder\'on operator. We also show that the Muckenhoup-type condition is necessary and sufficient for the boundedness on weighted local Morrey spaces of the usual Hardy-Littlewood maximal operator, simplifying the previous characterization of Nakamura-Sawano-Tanaka. For the same operator, in the case of global Morrey spaces the condition is necessary and for the sufficiency we add a local $A_p$ condition. We can extrapolate from Lebesgue $A_p$-weighted inequalities to weighted global and local Morrey spaces in a very general setting, with applications to many operators. 
\end{abstract}

\maketitle

\section{Introduction}\label{intro}

There are different versions of weighted Morrey spaces. We consider a quite general form by defining for $1\le p<\infty$ the space $\mathcal M^{p}(\varphi,w)$ as the class of measurable functions in $\rn$ for which  
\begin{equation}\label{normdef}
\|f\|_{\mathcal M^{p}(\varphi,w)} := \sup_{B}\left(\frac 1{\varphi(B)}\int_{B}|f|^p w\right)^{1/p}<\infty,
\end{equation}
where the supremum is taken over all the Euclidean balls in $\rn$. Here $w$ is a weight, that is,  a nonnegative locally integrable function, and $\varphi$ is a function defined on the set of balls in $\rn$ with values in $(0,\infty)$. This definition is essentially the same as the one for generalized weighted Morrey spaces considered by Guliyev in \cite{Gu12}, for instance. We are particularly interested in three distinguished cases of $\varphi$ corresponding to the weighted Morrey spaces introduced by N.~Samko (\cite{Sam09}), Komori-Shirai (\cite{KS09}) and Poelhuis-Torchinsky (\cite{PT15}), but also the scale of spaces introduced in \cite{DR20} will be considered. In general, we need to impose some conditions on $\varphi$ to avoid trivial spaces. We fix some conditions in Section \ref{bi}. 

There are many papers devoted to the study of the boundedness of operators in weighted Morrey spaces (see references in \cite{DR18} and \cite{NST19}, for instance). Nevertheless, to our knowledge it does not exist a characterization of the weights $w$ for which the Hardy-Littlewood maximal operator is bounded on $\mathcal M^{p}(\varphi,w)$, even in simple cases of $\varphi$. Sufficient conditions and necessary conditions are known, but not both simultaneously, see \cite{Ta15}, for instance. The desired characterization is obtained in \cite{NST19}  in the setting of weighted local Morrey spaces, that is, Morrey spaces in which only balls centered at the origin are considered in the definition of the norm (see Section \ref{zortzi}). 

In the case of weighted Lebesgue spaces the characterization of the boundedness of the Hardy-Littlewood maximal operator on $L^p(w)$ is given by the celebrated Muckenhoupt's $A_p$ condition. We consider here a related condition adapted to the Morrey spaces in the following form: a weight $w$ is in $A(\mathcal M^{p}(\varphi))$ if it satifies 
\begin{equation}\label{apmdef}
  [w]_{A(\mathcal M^{p}(\varphi))}:= 
	\sup_B \frac{\|\chi _B\|_{\mathcal M^{p}(\varphi,w)}\|\chi _B\|_{\mathcal M^{p}(\varphi,w)'}}{|B|} <\infty,
\end{equation}
where $\mathcal M^{p}(\varphi,w)'$ stands for the K\"othe dual of $\mathcal M^{p}(\varphi,w)$ (see the definition in Section \ref{bi}). Here $B$ runs over all the balls in $\rn$. Sometimes we restrict the set of balls and in such case we indicate it in the notation. Notice that replacing $\mathcal M^{p}(\varphi,w)$ with $L^p(w)$ we recover the usual $A_p$ condition, because the K\"othe dual of $L^p(w)$ is $L^{p'}(w^{1-p'})$ for $1<p<\infty$. A condition of this style is used in \cite{NST19} and is part of the necessary and sufficient condition for the boundedness of the Hardy-Littlewood maximal operator on local Morrey spaces. A condition involving the dual Morrey space (in a different version) appears also in \cite{Ta15}. The use of Muckenhoupt-type conditions in the Morrey setting was also suggested in \cite{Sam13}.

In this paper we first deal with the maximal operator defined only with balls centered at the origin, that is, 
\begin{equation}\label{centhl}
M_0f(x)=\sup_{r>|x|}\frac 1{|B(0,r)|}\int_{|y|<r}|f(y)|dy.
\end{equation}
The weights for which $M_0$ is bounded on the Lebesgue space $L^p(w)$ are described in \cite{DMO13} through a Muckenhoupt-type condition restricted to balls centered at the origin. Actually, most of the work in \cite{DMO13} is done on $(0,\infty)$ and the extension to $\rn$ is mentioned in Section 6 of that paper. In the Morrey setting we prove that the condition \eqref{apmdef} restricted to such balls centered at the origin is necessary and sufficient for the boundedness of $M_0$ on $\mathcal M^{p}(\varphi,w)$. We prove this theorem in Section \ref{hiru}, but for the necessity we present the result in the more general setting of a basis of balls in $\rn$. The proof of the theorem gives also a quantitative result: the norm of $M_0$ as an operator bounded on $\mathcal M^{p}(\varphi,w)$ is comparable to the constant of the weight given by \eqref{apmdef}. We point out that in the case $p=1$ there are nontrivial weights for which the boundedness on Morrey spaces holds. A similar fact was already observed for Lebesgue spaces in \cite{DMO13}.  

In Section \ref{lau} we work with the $n$-dimensional version of the Calder\'on operator. In $(0,\infty)$ this operator is defined as the sum of the Hardy operator and its adjoint  (and it is equivalent to the Hilbert operator). Its natural extension to $\rn$ is
\begin{equation}\label{calddef}
Sf(x)=\frac 1{|x|^n}\int_{|y|<|x|} f(y) dy + \int_{|y|>|x|} \frac {f(y)}{|y|^n} dy.
\end{equation}
This operator majorizes $M_0$ (up to a constant) and shares its $L^p(w)$ weighted properties for $1<p<\infty$ (\cite{DMO13}). In the Morrey setting we obtain a necessary and sufficient condition together with a sharp quantitative estimate of the operator norm. The condition is in general more restrictive than the one for $M_0$, and we show that at least for some choices of $\varphi$ they do not coincide. 

In Section \ref{bost} we consider the usual Hardy-Littlewood maximal operator, denoted as $M$ in the paper. It can be written as the sum of $M_0$ (given in \eqref{centhl}) and a local maximal operator in the sense of \cite{LS10} (defined below in \eqref{defmaxloc}). It is easy to give a sufficient condition for the boundedness of $M$ on $\mathcal M^{p}(\varphi,w)$, it is enough to add a local (Lebesgue-type) $A_p$ condition to the one obtained for $M_0$.

In Section \ref{zortzi} we deal with weighted local Morrey spaces. We improve the necessary and sufficient condition for the boundedness of $M$ obtained in \cite{NST19} in two ways: we prove that a condition similar to \eqref{apmdef} with local spaces is enough, instead of the two conditions involved in the result of \cite{NST19},  and we get the result also for weak estimates, including $p=1$. In the latter case we also obtain a sharp quantitative estimate in terms of the constant of the weight. We work independently with $M_0$  and the local Hardy-Littlewood maximal operator  and later we use the equivalence of $M$ with the sum of both operators. 

In Section \ref{sei} we prove an abstract result of the type of those studied in \cite{DR18, DR19, DR20}. We start with inequalities for pairs of functions satisfied under the hypotheses of the extrapolation theorem for $A_p$ weights in the Lebesgue setting and we extend such inequalities to the Morrey setting under conditions on the weight which resemble those obtained for the Calder\'on operator in Section \ref{lau}. The extrapolation-type results   are obtained also for the weighted local Morrey spaces, and in this setting they are new.   

In Section \ref{zazpi} we check that many previous results for weighted Morrey spaces are included in our result. More precisely, we compare the results in this paper with those in \cite{DR20}. A new result we deduce from $A(\mathcal M^{p}(\varphi))$ is that the boundedness of the maximal operator on weighted Morrey spaces of Samko type (defined with $\varphi(B)=r_B^{\lambda}$ for $0<\lambda<n$) implies a quantitative reverse doubling condition on $w$, namely,  
\begin{equation}\label{rhsamko}
\frac {w(B_1)}{w(B_2)}\lesssim \left( \frac {|B_1|}{|B_2|} \right)^{\frac {\lambda}{n}},
\end{equation}
for balls $B_1$ and $B_2$ such that $B_1\subset B_2$ (see Proposition \ref{rhprop}). This inequality is weaker than the reverse H\"older assumption introduced in our previous papers. The results in this section are also applied to the weighted local Morrey spaces, and in this setting many of them are new.

\subsection*{Notation} \begin{itemize}
\item Given a set $A$, $\chi_A$ is its characteristic function, $|A|$ is its Lebesgue measure, and for a weight $v$, $v(A)$ is the integral of $v$ over $A$ (but notice that $\varphi$ is not a weight in the definition of the norm of the Morrey space);  
\item for a ball $B$, $r_B$ will denote its radius and $c_B$ its center; for $\lambda>0$, $\lambda B$ denotes the ball with center $c_B$ and radius $\lambda r_B$; $\widetilde{B}$ is the smallest ball centered at the origin containing $B$, that is $B(0, |c_B|+r_B)$;
\item $F\lesssim G$ means that the inequality $F\le c\,G$ holds with a constant $c$ depending only on the ambient parameters; and $F\sim G$ means that both $F\lesssim G$ and $G\lesssim F$ hold.
\end{itemize}

%%%%%%%%%%%%%%%%%%%%%%%%%%%%  

\section{Preliminary results and significant properties of weighted Morrey spaces}\label{bi}

The most popular versions of weighted Morrey spaces considered in the literature are the following:
\begin{itemize}
\item $\varphi(B)=w(B)^{\lambda/n}$ ($0<\lambda<n$), introduced by Komori and Shirai in \cite{KS09};
\item  $\varphi(B)=r_B^{\lambda}$ ($0<\lambda<n$), introduced by N.~Samko in \cite{Sam09}.
\end{itemize}
It is worth mentioning the case $\varphi(B)=w(B) \phi(r_B)^{-1}$ considered by Poelhuis and Torchinsky in \cite{PT15}, which is a weighted version of the spaces in \cite{SST11}. Formally, the generalized weighted Morrey spaces of Guliyev in \cite{Gu12} are the same as the ones defined by \eqref{normdef} except for the notation: our $\varphi(B)$ is to be replaced by $w(B) \psi(B)^{-1}$. In \cite{DR20} we introduced a double-parameter family of weighted Morrey spaces by considering $\varphi(B)=r_B^{\lambda_1}w(B)^{\lambda_2/n}$, which contains the relevant examples, and discussed conditions on $\lambda_1$ and $\lambda_2$ for the non triviality of the space.  

\subsection{Doubling, reverse doubling and reverse H\"older}
We assume that $\varphi$ is doubling, that is,
\begin{equation*}
\varphi(2B)\lesssim \varphi(B).
\end{equation*}
We also assume that it satisfies a reverse doubling property of the form 
\begin{equation}\label{rdcond}
\frac {\varphi(B_1)}{\varphi(B_2)}\lesssim \left( \frac {|B_1|}{|B_2|} \right)^{\delta},
\end{equation}
for some $\delta>0$ and every pair of balls $B_1$ and $B_2$ such that $B_1\subset B_2$.  A doubling measure in $\rn$ always satisfies a reverse doubling property (see \cite{ST89}, for instance), but we are not assuming that $\varphi$ is a measure. Nevertheless, in the relevant examples it is easy to check that $\varphi$ satisfies a reverse doubling property. We write $\varphi\in RD_\delta$ to precise the exponent in \eqref{rdcond}. We write $RD_\delta(a)$ when condition \eqref{rdcond} is required only for balls centered at $a$.

From the reverse doubling property it follows that if $b<1$ the inequality  
\begin{equation}\label{rdplus}
\sum_{j=0}^\infty \varphi(B(0,b^jR))\lesssim  \varphi(B(0,R))
\end{equation}
holds. 

\begin{remark}\label{rm21}
Using these properties of $\varphi$ the following reductions in the norm \eqref{normdef} are obtained using \cite[Proposition 2.1]{DR20}:
 \begin{enumerate}
 \item [(a)] for $\varphi$ doubling we do not need to consider balls $B$ with $0<|c_B|<4r_B$;
 \item [(b)] if moreover $\varphi$ satisfies \eqref{rdplus}, then the supremum in \eqref{normdef} restricted to balls centered at the origin is equivalent to the supremum restricted to balls $B$ such that $|c_B|=4r_B$. 
 \end{enumerate}  
Using both properties together, it is enough to consider \eqref{normdef} with balls $B$ such that $|c_B|>4r_B$. For the local Morrey spaces of Section \ref{zortzi} $|c_B|=4r_B$ will be enough. 

\end{remark}

\begin{remark}\label{rm22}
Let $f$ be supported in a ball $B$. Then in \eqref{normdef} it is enough to take into account balls contained in $2B$. Indeed, if $Q$ is a ball such that $Q\cap B\neq\emptyset$ and $Q\cap(2B)^c\neq\emptyset$, then it is easy to check that $B\subset 5Q$. Therefore,
\begin{equation*}
\frac 1{\varphi(Q)}\int_{Q\cap B}|f|^p w \le \frac{\varphi(5Q)}{\varphi(Q)}\frac{\varphi(B)}{\varphi(5Q)}\frac 1{\varphi(B)}\int_{B}|f|^p w
\lesssim \frac 1{\varphi(B)}\int_{B}|f|^p w,
\end{equation*}
using the doubling and reverse doubling properties of $\varphi$.
\end{remark}

The weight $w$ satisfies a reverse H\"older inequality with index $\sigma$, denoted as $w\in RH_\sigma$, if 
\begin{equation*}%\label{TFE5}
  \left(\frac{1}{|B|}\int_{B} w(x)^\sigma dx \right)^\frac{1}{\sigma}\le \frac{C}{|B|} \int_{B} w(x) dx.
\end{equation*}
If $w\in RH_\sigma$, for any ball $B$ and any measurable $E\subset B$ it holds that 
\begin{equation}\label{ineqRH}
 \frac{w(E)}{w(B)}\le c \left(\frac{|E|}{|B|}\right)^{1/\sigma'}. 
\end{equation}
In particular, $w\in RD_{1/\sigma'}$. Nevertheless, a weight can be reverse doubling without satisfying a reverse H\"older property. Actually, it can be reverse doubling without even being doubling. For more information about the properties of being doubling, reverse doubling and reverse H\"older we refer to \cite{ST89}.

To make sense of condition \eqref{apmdef} we always assume that characteristic functions of balls are in the Morrey space $\mathcal M^{p}(\varphi,w)$. In the particular cases discussed in Section \ref{zazpi} this will be a consequence of the assumptions of the involved parameters.

\subsection{The K\"othe dual (associate space) of a Banach lattice} \label{kothesub}
Let $X$ be a Banach lattice of measurable functions in $\rn$ (a Banach space in which $|g|\le |f|$ a.e. and $f\in X$ implies $g\in X$ and $\|g\|_X\le \|f\|_X$). We define its K\"othe dual as the space $X'$ collecting the measurable functions $g$ such that
\begin{equation}\label{kothe}
\|g\|_{X'}:=\sup \left\{ \int_{\rn} |fg|: f\in X \text{ and } \|f\|_X\le 1\right\}<\infty.
\end{equation}
By the duality property of the Lebesgue spaces, the associate space of $L^p(w)$ is $L^{p'}(w^{1-p'})$ when $1<p<\infty$ and $p'$ is the exponent conjugate to $p$. See \cite{MST18} for a study of the K\"othe dual space of certain Morrey-type spaces. Associate spaces are often defined for Banach function spaces, but Morrey spaces need not be Banach function spaces in the sense of the definition given in \cite[Definitions 1.1 and 1.3]{BS88} as proved in \cite{ST15}. Nevertheless, the definition makes sense for Morrey spaces, which are Banach lattices.

An immediate consequence of \eqref{kothe} is the H\"older-type inequality
\begin{equation*}
\int_{\rn} |fg| \leq \|f\|_X \|g\|_{X'}.
\end{equation*}

The K\"othe bidual $X'':=(X')'$ coincides with $X$ isometrically if and only if $X$ has the Fatou property (\cite[Remark 2 in p.\ 30]{LT79}). This means that if $0\le f_n \uparrow f$ a.e. and $\|f_n\|_X\le 1$, then $f\in X$ and $\lim_{n\to\infty}\|f_n\|_X=\|f\|_X$. Morrey spaces satisfy the Fatou property. When $X''=X$, \eqref{kothe} holds switching the roles of $X$ and $X'$.

\subsection{Weak-type Morrey spaces}\label{subs23}

We define for $1\le p<\infty$ the weighted weak-type Morrey space $W\mathcal M^{p}(\varphi,w)$ as the collection of measurable functions in $\rn$ for which  
\begin{equation}\label{wnormdef}
\aligned
\|f\|_{W\mathcal M^{p}(\varphi,w)} &:= \sup_{B}\frac {\|f\chi_B\|_{L^{p,\infty}(w)}}{\varphi(B)^{1/p}}\\
& = \sup_{B}\sup_{t>0}\frac {t (w(\{x\in B: |f(x)|>t\})^{1/p}}{\varphi(B)^{1/p}}<\infty,
\endaligned
\end{equation}
where the second equality is the definition of the quasinorm in the Lorentz space $L^{p,\infty}(w)$. It is clear that $\mathcal M^{p}(\varphi,w)$ is continuously embedded in $W\mathcal M^{p}(\varphi,w)$. Also that for a measurable set $E$, it holds
\begin{equation}\label{wcharac}
\|\chi_E\|_{W\mathcal M^{p}(\varphi,w)} = \|\chi_E\|_{\mathcal M^{p}(\varphi,w)}.
\end{equation}

\subsection{The local Hardy-Littlewood maximal operator}

For fixed $\kappa\in (0,1)$ we define the basis $\mathcal{B}_{\kappa,\text{loc}}$ as the family of balls $B$ such that $r_B<\kappa |c_B|$. The local Hardy-Littlewood maximal operator associated to $\mathcal{B}_{\kappa,\text{loc}}$ acting on $f$ at $x\ne 0$ is defined as
\begin{equation}\label{defmaxloc}
M_{\kappa,\text{loc}}f(x)=\sup_{x\in \mathcal{B}_{\kappa,\text{loc}}}\frac 1{|B|}\int_{B}|f|.
\end{equation}
C.-C.~Lin and K.~Stempak considered in \cite{LS10} a variant of this operator and characterized its weighted inequalities by a local $A_p$ condition. E. Harboure, O. Salinas, and B. Viviani obtained in \cite{HSV14} a similar result in a more general setting. Their result applied to the operator in \eqref{defmaxloc} says the following:  $M_{\kappa,\text{loc}}$ is bounded on $L^p(w)$ ($1<p<\infty$) if and only if $w$ satisfies the $A_{p,\text{loc}}$ condition 
\begin{equation}\label{defaploc}
[w]_{A_{p,\text{loc}}}=\sup_B \frac{w(B)^{1/p}\,w^{1-p'}(B)^{1/p'}}{|B|}<\infty,
\end{equation}
where the supremum is taken over the balls in $\mathcal{B}_{\kappa,\text{loc}}$. (Usually the $A_p$ constant of the weight $w$ is defined as the $p$-th power of the constant in \eqref{defaploc}.) They also proved that the condition is independent of $\kappa$, so that we do not use this parameter in the notation of the weight class. It is clear that the usual $A_p$ weights are in $A_{p,\text{loc}}$ and it is immediate to check that $|x|^\alpha w(x)\in A_{p,\text{loc}}$ whenever $w\in A_p$ and $\alpha\in \R$. The class $A_{1,\text{loc}}$ is defined with the appropriate modification as usual and characterizes the weighted weak $(1,1)$ type inequality for $M_{\kappa,\text{loc}}$. 

Notice that using the concept of K\"othe dual the $A_p$ condition can be written as 
\begin{equation*} 
	\sup_B \frac{\|\chi _B\|_{L^{p}(w)}\|\chi _B\|_{L^{p}(w)'}}{|B|}<\infty.
\end{equation*}
Under this form it is valid also for $p=1$ because $\|g\|_{L^1(w)'}=\|gw^{-1}\|_\infty$.

We warn the reader that there is another version of local $A_p$ weights. It is associated to the Hardy-Littlewood maximal operator restricted to balls of radius less than $1$ and produces a class of weights which is different from the one we are dealing with here. See \cite{Ry01} for details.

\subsection{Inclusion properties of $A(\mathcal M^{p}(\varphi))$}

The classes $A(\mathcal M^{p}(\varphi))$ are increasing in $p$, like the usual $A_p$ classes.
\begin{proposition}
 Let $1\le q<p<\infty$. Then $A(\mathcal M^{q}(\varphi))\subset A(\mathcal M^{p}(\varphi))$.
\end{proposition}

\begin{proof}
 Let $s=p/q$. The property
  \begin{equation}\label{normscale}
\|f^s\|_{\mathcal M^{q}(\varphi,w)} = \|f\|_{\mathcal M^{p}(\varphi,w)}^s
\end{equation}
is immediate from the definition of the norm. In particular,
 \begin{equation*}
\|\chi _B\|_{\mathcal M^{q}(\varphi,w)} = \|\chi _B\|_{\mathcal M^{p}(\varphi,w)}^s
\end{equation*}
for any ball $B$.

We also have
\begin{equation*}
\int_B f\le \left(\int_B f^s\right)^{1/s} |B|^{1/s'}\le \|f^s\|_{\mathcal M^{q}(\varphi,w)}^{1/s} \|\chi _B\|_{\mathcal M^{q}(\varphi,w)'}^{1/s} |B|^{1/s'}.
\end{equation*}
Taking the supremum over $f\in \mathcal M^{p}(\varphi,w)$ with norm $\le 1$ and using \eqref{normscale} we deduce
 \begin{equation*}
\|\chi _B\|_{\mathcal M^{p}(\varphi,w)'}\le \|\chi _B\|_{\mathcal M^{q}(\varphi,w)'}^{1/s} |B|^{1/s'}.
\end{equation*}
Therefore,
 \begin{equation*}
\frac{\|\chi _B\|_{\mathcal M^{p}(\varphi,w)}\|\chi _B\|_{\mathcal M^{p}(\varphi,w)'}}{|B|} \le  \left(\frac{\|\chi _B\|_{\mathcal M^{q}(\varphi,w)}\|\chi _B\|_{\mathcal M^{q}(\varphi,w)'}}{|B|}\right)^{1/s}
\end{equation*}
and the proposition is proved.
\end{proof}

%%%%%%%%%%%%%%%%%%%%%%%%%%%%%
\section{The maximal operator on balls centered at the origin}\label{hiru}

\subsection{Necessity of the $A_p$-type Morrey condition}

We prove this result in a more general setting by using a standard technique. Let $\mathcal B$ a collection of balls in $\rn$. For simplicity we assume that their union is $\rn$. Define the maximal operator associated to $\mathcal B$ as
\begin{equation}\label{maxbasis}
M_{\mathcal B}f(x)=\sup_{x\in B, B\in \mathcal B}\frac 1{|B|}\int_{B}|f|.
\end{equation}
Denote as $A_{\mathcal B}(\mathcal M^{p}(\varphi))$ the class of weights satisfying \eqref{apmdef} with the supremum taken on the balls of $\mathcal B$ and let $[w]_{A_{\mathcal B}(\mathcal M^{p}(\varphi))}$ be the constant of the weight $w$.

\begin{theorem}\label{necbasis}
 Let $1\le p<\infty$. If $M_{\mathcal B}$ is bounded from $\mathcal M^{p}(\varphi,w)$ to $W\mathcal M^{p}(\varphi,w)$ (hence also if it is bounded on $\mathcal M^{p}(\varphi,w)$), then $w\in A_{\mathcal B}(\mathcal M^{p}(\varphi))$. Moreover, the operator norm of $M_{\mathcal B}$ is bounded below by $[w]_{A_{\mathcal B}(\mathcal M^{p}(\varphi))}$.
\end{theorem}

\begin{proof}
 Fix a ball $B\in \mathcal B$. Then for $x\in \rn$
 \begin{equation*}
\left(\frac 1{|B|}\int_{B}|f|\right) \chi_B(x)\le M_{\mathcal B}f(x).
\end{equation*}
Hence,
\begin{equation*}
\left(\frac 1{|B|}\int_{\rn}\chi_B|f|\right) \|\chi_B\|_{W\mathcal M^{p}(\varphi,w)}\le \|M_{\mathcal B}f\|_{W\mathcal M^{p}(\varphi,w)}\le C \|f\|_{\mathcal M^{p}(\varphi,w)}.
\end{equation*}
Taking the supremum over the functions $f$ for which  $\|f\|_{\mathcal M^{p}(\varphi,w)}\le 1$ and using \eqref{wcharac}, we obtain that $w\in A_{\mathcal B}(\mathcal M^{p}(\varphi))$ and that the operator norm of $M_{\mathcal B}$ is bounded below by $[w]_{A_{\mathcal B}(\mathcal M^{p}(\varphi))}$.
\end{proof}

\subsection{Sufficiency of the $A_p$-type Morrey condition for $M_0$}

The operator $M_0$ has been defined in \eqref{centhl} as the maximal operator associated to the basis $\mathcal B_0$ of balls centered at the origin. We use only the subscript $0$ and not $\mathcal B_0$, and denote as $A_{0}(\mathcal M^{p}(\varphi))$ and $[w]_{A_{0}(\mathcal M^{p}(\varphi))}$ the corresponding class of weights and the constant of $w$.

\begin{theorem}\label{teocero}
  Let $1\le p<\infty$. The following are equivalent:
 \begin{itemize}
\item [(i)] The operator $M_0$ is bounded from $\mathcal M^{p}(\varphi,w)$ to $W\mathcal M^{p}(\varphi,w)$.
\item [(ii)] $w\in A_{0}(\mathcal M^{p}(\varphi))$.
\item [(iii)] The operator $M_0$ is bounded on $\mathcal M^{p}(\varphi,w)$.
\end{itemize}
Moreover, the operator norm of $M_0$ in (i) and (iii) is equivalent to $[w]_{A_{0}(\mathcal M^{p}(\varphi))}$.
\end{theorem}

\begin{proof}
We only need to prove that (ii) implies (iii). Indeed, Theorem \ref{necbasis} says that (i) implies (ii), and (iii) implies (i) due to the embedding of $\mathcal M^{p}(\varphi,w)$ into $W\mathcal M^{p}(\varphi,w)$.

 Let $B$ be a ball such that $|c_B|>4r_B$ and $f$ a locally integrable function. By definition, $M_0f$ is radially decreasing. As a consequence, it is essentially constant in $B$ in the sense that
 \begin{equation*}
\sup_{y\in B} M_0f(y) \le C \inf_{y\in B} M_0f(y),
\end{equation*}
where $C$ is a dimensional constant. Indeed, to compute $M_0f(y)$ we only need to take into account balls $B(0,\rho)$ with $\rho>|c_B|-r_B$. But $|c_B|+r_B<2(|c_B|-r_B)$ and the average of $f$ on  $B(0,\rho)$ with $|c_B|-r_B<\rho<|c_B|+r_B$ is bounded by a constant times the average on $B(0,|c_B|+r_B)$. Hence, for $y\in B$,
 \begin{equation*}
M_0f(y) \sim  \sup_{\rho>|c_B|+r_B}\frac 1{|B(0,\rho)|}\int_{B(0,\rho)}|f|.
\end{equation*} 
 Let $B'$ be a ball centered at the origin and containing $B$. Then  
\begin{equation*}
\aligned
 \frac 1{|B'|}&\int_{B'}|f|\le 
\frac 1{|B'|} \|f\|_{\mathcal M^{p}(\varphi,w)}  \|\chi _{B'}\|_{\mathcal M^{p}(\varphi,w)'}\\
 & \le \|f\|_{\mathcal M^{p}(\varphi,w)} \frac{[w]_{A_{0}(\mathcal M^{p}(\varphi))}}{\|\chi _{B'}\|_{\mathcal M^{p}(\varphi,w)}}
 \le \|f\|_{\mathcal M^{p}(\varphi,w)} \frac{[w]_{A_{0}(\mathcal M^{p}(\varphi))}}{\|\chi _{B}\|_{\mathcal M^{p}(\varphi,w)}}.
\endaligned
\end{equation*}
(We used the trivial inequality $\|\chi _B\|_{\mathcal M^{p}(\varphi,w)}\le \|\chi _{B'}\|_{\mathcal M^{p}(\varphi,w)}$.) As a consequence, 
\begin{equation*}
\sup_{y\in B} M_0f(y) \lesssim  \|f\|_{\mathcal M^{p}(\varphi,w)} \frac{[w]_{A_{0}(\mathcal M^{p}(\varphi))}}{\|\chi _{B}\|_{\mathcal M^{p}(\varphi,w)}}.
\end{equation*} 
Using this estimate,
\begin{equation*}
\aligned
& \left(\frac 1{\varphi(B)}\int_B (M_0f)^p w\right)^{1/p}\\
&\qquad  \lesssim \left(\frac 1{\varphi(B)}\int_B w\right)^{1/p} \|f\|_{\mathcal M^{p}(\varphi,w)} \frac{[w]_{A_{0}(\mathcal M^{p}(\varphi))}}{\|\chi _{B}\|_{\mathcal M^{p}(\varphi,w)}} \\
&\qquad \lesssim \|f\|_{\mathcal M^{p}(\varphi,w)} [w]_{A_{0}(\mathcal M^{p}(\varphi))}.
\endaligned
\end{equation*}
This proves the sufficiency of $w\in A_{0}(\mathcal M^{p}(\varphi))$ and also that the operator norm of $M_0$ on $\mathcal M^{p}(\varphi,w)$ is bounded above by $C[w]_{A_0(\mathcal M^{p}(\varphi))}$. 
\end{proof}

\begin{remark}
Notice that the proof shows that for any ball $B'$ centered at the origin the integral of $f$ over $B'$ is finite for $f\in \mathcal M^{p}(\varphi,w)$ and $w\in A_{0}(\mathcal M^{p}(\varphi))$. Hence, assuming in the proof that $f$ is locally integrable is not an extra condition. We need the local integrability of $f$ to define the Hardy-Littlewood maximal operator.
\end{remark}

\begin{remark}
The result in Theorem \ref{teocero} is valid for $p=1$. It was proved in \cite{DMO13} that for nontrivial weights $w$, the operator $M_0$ is bounded on $L^1(w)$, although it is unbounded on the unweighted space $L^1$. Also in the Morrey setting nontrivial results can be established for $p=1$, even with $w\equiv 1$. On the one hand, we have the trivial estimate 
\begin{equation*}
\|\chi _{B}\|_{\mathcal M^{1}(\varphi,1)'}\le \varphi(B).
\end{equation*}
On the other hand, in many cases it holds
\begin{equation*}
\|\chi _{B}\|_{\mathcal M^{1}(\varphi,1)}\lesssim \frac{|B|}{\varphi (B)}.
\end{equation*}
This holds in particular for the usual Morrey spaces, where $\varphi(B)=r_B^\lambda$ with $0<\lambda<n$.

Notice also that the characterization of the weighted inequalities in the strong case and in the weak case are the same for all values of $p$. This holds also in the Lebesgue setting for $p>1$, but not for $p=1$ (see  \cite{DMO13}).
\end{remark}

%%%%%%%%%%%%%%%%%%%%%%%%%%%%%
\section{The Calder\'on operator}\label{lau}

The $n$-dimensional version of the Calder\'on operator has been defined in \eqref{calddef}. It is pointwise equivalent to the $n$-dimensional version of the Hilbert operator (not to be confused with the Hilbert transform) defined as
\begin{equation*}
\widetilde S f(x)=\int_{\rn}\frac {f(y)}{|x|^n+|y|^n}dy.
\end{equation*}
These operators majorize pointwise (up to a constant) $M_0$ for nonnegative $f$, and share its weighted boundedness properties on $L^p(w)$ (see \cite{DMO13}).

\begin{theorem}\label{morcal}
Let $1\le p<\infty$. The operator $S$ is bounded on $\mathcal M^{p}(\varphi,w)$ if and only if
\begin{equation}\label{calsuf}
\sup_B \frac 1{|B|} \|\chi _B\|_{\mathcal M^{p}(\varphi,w)}  \|M_0(\chi _{B})\|_{\mathcal M^{p}(\varphi,w)'} <\infty,
\end{equation}
where the supremum is taken over the balls centered at the origin. Moreover, the operator norm of $S$ is comparable to the left-hand side of \eqref{calsuf}.
\end{theorem}

\begin{proof}
\textit{Necessity.} Let $B$ be a ball centered at the origin. Then
\begin{equation*}
M_0(\chi _{B})(y)\sim |B| \min \left(\frac 1{|B|}, \frac 1{|y|^n}\right).
\end{equation*}
Hence,
\begin{equation*}
\aligned
\frac 1{|B|}\left(\int_{\rn} M_0(\chi _{B}) |f|\right)\chi_B(x)&\lesssim \int_{\rn} |f(y)|\min  \left(\frac 1{|x|^n}, \frac 1{|y|^n}\right)dy \\
&=S(|f|)(x).
\endaligned
\end{equation*}
Assuming that $S$ is bounded on $\mathcal M^{p}(\varphi,w)$ we have
\begin{equation*}
\frac 1{|B|}\left(\int_{\rn} M_0(\chi _{B}) |f|\right)\|\chi_B\|_{\mathcal M^{p}(\varphi,w)}\lesssim \|S(|f|)\|_{\mathcal M^{p}(\varphi,w)}\lesssim \|f\|_{\mathcal M^{p}(\varphi,w)}.
\end{equation*}
Taking the supremum over the functions $f$ with $\|f\|_{\mathcal M^{p}(\varphi,w)}\le 1$ we obtain \eqref{calsuf}.

\textit{Sufficiency.} Let $B$ be a ball such that $|c_B|>4r_B$. Without loss of generality we assume that $f$ is nonnegative. By an argument similar to the one in Theorem \ref{teocero} we deduce that the value of $Sf(x)$ for $x\in B$ is essentially constant, namely,  
\begin{equation*}
Sf(x)\sim \frac 1{|\widetilde{B}|}\int_{|y|<R}f(y)dy+ \int_{|y|>R}\frac {f(y)}{|y|^n}dy
\sim \int_{\rn}  \frac {M_0(\chi_{\widetilde{B}})}{|\widetilde{B}|} \,f,
\end{equation*}
where $\widetilde{B}$ is the smallest ball centered at the origin containing $B$. 
 Then we have 
 \begin{equation*}
 \aligned
\left(\frac 1{\varphi(B)}\int_B |Sf|^p w\right)^{1/p} 
&\lesssim \left(\frac 1{\varphi(B)}\int_B w\right)^{1/p} \int_{\rn} \frac {M_0(\chi_{\widetilde{B}})}{|\widetilde{B}|}\,f\\
&\lesssim  \|\chi _B\|_{\mathcal M^{p}(\varphi,w)} \frac {\|M_0(\chi_{\widetilde{B}})\|_{\mathcal M^{p}(\varphi,w)'}}{|\widetilde{B}|} \|f\|_{\mathcal M^{p}(\varphi,w)}.
\endaligned
\end{equation*}
Replacing $B$ with $\widetilde{B}$ in the right-hand side we deduce that condition \eqref{calsuf} is sufficient.
\end{proof}

\begin{remark}\label{rem42}
It is clear that \eqref{calsuf} implies $A_{0}(\mathcal M^{p}(\varphi))$. This can be also deduced from the pointwise  inequality $M_0f(x)\le S(|f|)(x)$. But the condition $A_{0}(\mathcal M^{p}(\varphi))$ of Theorem \ref{teocero} is not sufficient in general for the boundedness of the Calder\'on operator on $\mathcal M^{p}(\varphi,w)$.  Let us consider the case $\varphi(B)=r(B)^\lambda$, corresponding to the spaces of Samko type. Then $w(x)=|x|^{\lambda-n}\in A_{0}(\mathcal M^{p}(\varphi))$, but $S$ is not bounded on $\mathcal M^{p}(\varphi,w)$. 
 
 To show that  $|x|^{\lambda-n}\in A_{0}(\mathcal M^{p}(\varphi))$ it would be enough to invoke Tanaka's result in \cite{Ta15} saying that $M$ (which is bigger than $M_0$) is bounded for such weighted Morrey space (this result is also proved in \cite{DR19}). But we can also check it directly. To this end, let $B$ be a ball centered at the origin and $Q$ a ball with $|c_Q|>4r_Q$. Then
\begin{equation*}
\frac 1{r^\lambda} \int_{Q\cap B} |y|^{\lambda -n} dy\lesssim \left(\frac {r_Q}{|c_Q|}\right)^{n-\lambda}\le 1. 
\end{equation*}
Hence, $\|\chi_B\|_{\mathcal M^{p}(\varphi,w)}\lesssim 1$.

On the other hand, 
\begin{equation*}
\aligned
\int_B f & \le \left(\frac 1{r_B^\lambda} \int_B |f(y)|^p  |y|^{\lambda -n} dy\right)^{1/p} r_B^{\lambda/p} \left(\int_B |y|^{(n-\lambda)(p'-1)} dy\right)^{1/p'} \\
& \lesssim \|f\|_{\mathcal M^{p}(\varphi,w)} \, r_B^{n} \sim \|f\|_{\mathcal M^{p}(\varphi,w)} \, |B|.
\endaligned
\end{equation*}
Taking the supremum over the functions $f$ such that $\|f\|_{\mathcal M^{p}(\varphi,w)}\le 1$, we deduce that $\|\chi_B\|_{\mathcal M^{p}(\varphi,w)'}\lesssim |B|$, and the condition $A_{0}(\mathcal M^{p}(\varphi))$ is satisfied.

Consider now $B=B(0,1)$. For $|x|<1$, 
\begin{equation*}
S(\chi_B)(x)= \frac 1{|x|^n}\int_{|y|<|x|}dy+ \int_{|x|<|y|<1}\frac {1}{|y|^n}dy\sim 1-\log |x|.
\end{equation*}
But $S(\chi_B)\notin \mathcal M^{p}(\varphi,w)$. Indeed, 
\begin{equation*}
\left(\frac 1{(|x|/4)^\lambda} \int_{B(x,|x|/4)} |1-\log |y||^p |y|^{\lambda -n} dy\right)^{1/p}\sim |1-\log |x||
\end{equation*}
is unbounded when $x$ is close to $0$.

Notice that if $M_0$ is bounded on $\mathcal M^{p}(\varphi,w)'$, then \eqref{calsuf} is equivalent to $w\in A_{0}(\mathcal M^{p}(\varphi))$. This counterexample shows that in the case $\varphi(B)=r(B)^\lambda$ and $w(x)=|x|^{\lambda-n}$, $M_0$ is bounded on $\mathcal M^{p}(\varphi,w)$, but not on $\mathcal M^{p}(\varphi,w)'$.
\end{remark}

The last conclusion can also be deduced by means of a result similar to those in \cite{Ru18}, which is easily obtained in this setting: \textit{$M_0$ is bounded on $X$ and $X'$ if and only if $S$ is bounded on $X$, whenever $X$ has the Fatou property}. This is a consequence of the following:
\begin{itemize}
\item [(i)] $S$ bounded on $X$ is equivalent to $S$ bounded on $X'$, because $S$ is self-adjoint;
\item [(ii)] $M_0f\le Sf$ pointwise for nonnegative $f$;
\item [(iii)] $S$ is the sum of the Hardy operator and its adjoint, and $M_0$ is a pointwise bound for the Hardy operator.
\end{itemize}
Based on this we deduce the following two results: 
\begin{enumerate}
\item If $M_0$ is bounded on $\mathcal M^{p}(\varphi,w)$ and $S$ is not, then $M_0$ is not bounded on $\mathcal M^{p}(\varphi,w)'$.
\item If $w\in A_{0}(\mathcal M^{p}(\varphi))$ and 
\begin{equation*}
\|M_0(\chi _{B})\|_{\mathcal M^{p}(\varphi,w)'} \lesssim \|\chi _{B}\|_{\mathcal M^{p}(\varphi,w)'} 
\end{equation*}
for balls centered at the origin, then $M_0$ is bounded on $\mathcal M^{p}(\varphi,w)'$.
\end{enumerate}

%%%%%%%%%%%%%%%%%%%%%%%%%%%%%
\section{The Hardy-Littlewood maximal operator}\label{bost}

In this section we deal with the full Hardy-Littlewood maximal operator. We prove that it is bounded on  $\mathcal M^{p}(\varphi,w)$ if the weight satisfies  \eqref{apmdef} and a local $A_p$ condition \eqref{defaploc}. In Section \ref{zazpi} we show that many weights satisfy this joint sufficient condition.

\begin{theorem}\label{teohl}
The following results hold for $M$. 
\begin{enumerate}
\item  Let $1\le p<\infty$. If $M$ is bounded from $\mathcal M^{p}(\varphi,w)$ to $W\mathcal M^{p}(\varphi,w)$, then $w\in A(\mathcal M^{p}(\varphi))$. 
\item Let $1< p<\infty$. If $w\in A(\mathcal M^{p}(\varphi))$ and $w(\cdot + a)\in A_{p,\textrm{loc}}(\rn)$ for some $a\in \rn$, then $M$ is bounded on $\mathcal M^{p}(\varphi,w)$.
\item If $w\in A(\mathcal M^{1}(\varphi))$ and $w(\cdot + a)\in A_{1,\textrm{loc}}(\rn)$ for some $a\in \rn$, then $M$ is bounded from $\mathcal M^{1}(\varphi,w)$ to $W\mathcal M^{1}(\varphi,w)$.
\end{enumerate}
\end{theorem}

\begin{proof} 
(1) The necessity of $w\in A(\mathcal M^{p}(\varphi))$ is a particular case of Theorem \ref{necbasis}.

(2) We assume that $a=0$. Let $B$ with $|c_B|>4r_B$. Given a nonnegative function $f$ we decompose it as $f=f_1+f_2$, where $f_1=f\chi_{2B}$. Using the subadditivity of $M$ we have
\begin{equation*}
Mf(y)\le Mf_1(y)+Mf_2(y).
\end{equation*}
 
Since $f_1$ is supported in $2B$, to compute $Mf_1(y)$ for $y\in B$ we only need to consider balls $B(y,\rho)$ with $\rho<3r_B$. Moreover, for $r_B<\rho<3r_B$ the average on $B(y,\rho)$ is comparable to the average on $2B$. Therefore,
\begin{equation*}
\aligned
Mf_1(y)&\sim \sup_{\rho \le r_B} \frac 1{|B(y,\rho)|}\int_{B(y,\rho)}|f|+\frac 1{|2B|}\int_{2B}|f|\\
&\sim M_{\textrm{loc}}f_1(y)+\frac 1{|2B|}\int_{2B}|f|.
\endaligned
\end{equation*}

On the other hand, since $f_2$ is supported on $(2B)^c$, to compute $Mf_2(y)$ for $y\in B$ one  only needs to consider balls of radius greater than $r_B$ and as a consequence, $Mf_2$ is almost constant on $B$. Then \begin{equation*}
Mf_2(y)\sim \sup_{B' \supset B}\frac 1{|B'|}\int_{B'}|f|\quad \text{ for all } y\in B.
\end{equation*} 

Altogether, we have for all $y\in B$ and some $B'$ with $B\subset B'$ that 
\begin{equation}\label{equivm}
Mf(y)\sim M_{\textrm{loc}}(f \chi_{2B})(y) + \sup_{B' \supset B} \frac 1{|B'|}\int_{B'}|f|.
\end{equation}

Then we have
\begin{equation*}
\aligned
&\left(\frac 1{\varphi(B)}\int_{B}|Mf|^p w\right)^{1/p} \\
&\lesssim 
\left(\frac 1{\varphi(B)}\int_{B}|M_{\textrm{loc}}(f\chi_{2B})|^p w\right)^{1/p} + \left(\frac 1{\varphi(B)}\int_{B}w\right)^{1/p} \sup_{B' \supset B}\frac 1{|B'|}\int_{B'}|f|.
\endaligned
\end{equation*} 
If $w\in A_{p,\textrm{loc}}(\rn)$, the first term is bounded by a constant times
\begin{equation*}
\left(\frac 1{\varphi(B)}\int_{2B}|f|^p w\right)^{1/p} \le C \|f\|_{\mathcal M^{p}(\varphi,w)}, 
\end{equation*} 
using the doubling property of $\varphi$. The second term is treated as in the proof of Theorem \ref{teocero}. 

If $a\ne 0$, it is enough to make a translation.

(3) From \eqref{equivm} we use the weak-type (1,1) inequality of $M_{\textrm{loc}}$. The last term is treated as for $p>1$. 
\end{proof}

\begin{remark}
It is possible to generalize this statement as follows: write $w=\sum_{j=1}^N w_j$ and assume that for each $j$, $w_j(\cdot +a_j)$ satisfies the second condition, possibly for different values of $a_j$.
\end{remark}

%%%%%%%%%%%%%%%%%%%%%%%%%%%%%%%%%%%%%%%%%%%%%%%%%%%%%%%
\section{The Hardy-Littlewood maximal operator on weighted local Morrey spaces}\label{zortzi}

In this section we consider local Morrey spaces. We define for $1\le p<\infty$ the space $\mathcal {LM}^{p}(\varphi,w)$ as the class of measurable functions in $\rn$ for which  
\begin{equation}\label{normloc}
\|f\|_{\mathcal {LM}^{p}(\varphi,w)} := \sup_{R>0}\left(\frac 1{\varphi(B(0,R))}\int_{B(0,R)}|f|^p w\right)^{1/p}<\infty.
\end{equation} 
This is the same as \eqref{normdef}, but restricted to balls centered at the origin.  The local Morrey spaces are also named central Morrey spaces by some authors. 

We also consider the corresponding weak-type spaces. We define $W\mathcal {LM}^{p}(\varphi,w)$ by modifying the definition \eqref{normloc} as we did in Subsection \ref{subs23}. The embedding of $\mathcal {LM}^{p}(\varphi,w)$ into $W\mathcal {LM}^{p}(\varphi,w)$ is clear, and the equality \eqref{wcharac} also holds with local Morrey spaces.

Notice that here the term local does not refer to balls separated from the origin as in the case of the local Hardy-Littlewood maximal operator defined in Section \ref{bi}. We also point out that the term local Morrey space has been used to designate a different class of Morrey spaces, namely those obtained by restricting the supremum in \eqref{normdef} to balls of radius less than $1$. See \cite{RT15}, for instance.

Following part (b) of Remark \ref{rm21}, in the norm given by \eqref{normloc} we can replace the balls centered at the origin by balls $B$ with $|c_B|=4r_B$. 
The idea of this reduction comes from \cite{NST19}, where it is applied to dyadic cubes instead of balls.

As we did for the usual Morrey spaces, we consider the K\"othe dual of $\mathcal {LM}^{p}(\varphi,w)$ and define the class $A(\mathcal {LM}^{p}(\varphi))$ as in \eqref{apmdef}, using the local Morrey spaces instead, that is, 
\begin{equation}\label{apmlocdef}
  [w]_{A(\mathcal {LM}^{p}(\varphi))}:= 
	\sup_B \frac{\|\chi _B\|_{\mathcal {LM}^{p}(\varphi,w)}\|\chi _B\|_{\mathcal {LM}^{p}(\varphi,w)'}}{|B|} <\infty.
\end{equation}
Here the supremum is over all the balls in $\rn$. We denote as $A_0(\mathcal {LM}^{p}(\varphi))$ and $A_{\textrm{loc}}(\mathcal {LM}^{p}(\varphi))$ the classes of weights for which \eqref{apmlocdef} is satisfied restricted to balls centered at the origin and to balls $B$ such that $|c_B|>4r_B$, respectively. The corresponding constants will be $[w]_{A_0(\mathcal {LM}^{p}(\varphi))}$ and $[w]_{A_{\textrm{loc}}(\mathcal {LM}^{p}(\varphi))}$.

We prove that being in $A(\mathcal {LM}^{p}(\varphi))$ is necessary and sufficient for the boundedness of the Hardy-Littlewood maximal operator on weighted local Morrey spaces. This result simplifies the characterization obtained in \cite{NST19}, where two conditions were needed. Moreover, we extend the results to the setting of weak-type weighted inequalities, which were not considered in \cite{NST19}. In particular, we can cover the case $p=1$. In the weak case we also get a sharp quantitative result.

\begin{theorem}\label{mainthmloc}
(a) Let $1<p<\infty$. The Hardy-Littlewood maximal operator is bounded on $\mathcal {LM}^{p}(\varphi,w)$ if and only if $w\in A(\mathcal {LM}^{p}(\varphi))$.  

(b) Let $1\le p<\infty$. The Hardy-Littlewood maximal operator is bounded from $\mathcal {LM}^{p}(\varphi,w)$ to $W\mathcal {LM}^{p}(\varphi,w)$ if and only if $w\in A(\mathcal {LM}^{p}(\varphi))$. Moreover, the operator norm is comparable to $[w]_{A(\mathcal {LM}^{p}(\varphi))}$.
\end{theorem}

In the definition of the usual Hardy-Littlewood maximal operator balls of the form $B(y,r)$ with $0<|y|\le 4r$ can be avoided. This is because we can replace such a ball by $B(0,5r)$, which contains $B(y,r)$ and has comparable size. The remaining balls can be distributed into two sets: those centered at the origin and those of the form  $B(y,r)$ with $|y|>4r$. Denoting as $M_{\text{loc}}$ the local maximal operator $M_{1/4,\text{loc}}$ defined in \eqref{defmaxloc} we have
\begin{equation}\label{maxdec}
Mf(x) \sim M_0f(x) + M_{\text{loc}}f(x).
\end{equation}
To prove Theorem \ref{mainthmloc} we prove separate results for $M_0$ and $M_{\text{loc}}$.

We use the following key observation: for a function $f$ supported in a ball $B$ with $4r_B<|c_B|$ it holds that
\begin{equation}\label{locloc}
\|f\|_{\mathcal {LM}^{p}(\varphi,w)}\sim \left(\frac{1}{\varphi(\widetilde{B})}\int_B |f|^pw\right)^{1/p},
\end{equation}
where $\widetilde{B}$ is as defined in the introduction.  This is because the balls $B(0,R)$ with $R<|c_B|-r_B$ do not intersect $B$ and $\varphi(B(0,R))$ is essentially constant for $|c_B|-r_B<R<|c_B|+r_B$, due to the doubling property of $\varphi$.  In particular, if $f$ is the characteristic function of $B$,
\begin{equation}\label{locchar}
\|\chi_B\|_{\mathcal {LM}^{p}(\varphi,w)}=\|\chi_B\|_{W\mathcal {LM}^{p}(\varphi,w)}\sim \left(\frac{w(B)}{\varphi(\widetilde{B})}\right)^{1/p}.
\end{equation} 

Before characterizing the weights for $M_{\text{loc}}$ we prove the following lemma.

\begin{lemma}
Let $1\le p<\infty$. The class of weights $A_{\textrm{loc}}(\mathcal {LM}^{p}(\varphi))$ coincides with $A_{p,\text{loc}}$. Moreover, the constants $[w]_{A_{\textrm{loc}}(\mathcal {LM}^{p}(\varphi))}$ and  $[w]_{A_{p,\text{loc}}}$ are comparable.
\end{lemma}

\begin{proof}
Let $1<p<\infty$. We claim that 
\begin{equation}\label{normprime}
\|\chi_B\|_{\mathcal {LM}^{p}(\varphi,w)'}\sim w^{1-p'}(B)^{1/p'}\varphi(\widetilde{B})^{1/p}.
\end{equation}
Once this is proved, together with \eqref{locchar} we obtain the result of the lemma for $p>1$.

On the one hand we have 
\begin{equation*}
\|\chi_B\|_{\mathcal {LM}^{p}(\varphi,w)'}\ge \frac{1}{\|w^{1-p'}\chi_B\|_{\mathcal {LM}^{p}(\varphi,w)}}\int_B w^{1-p'} \sim
 w^{1-p'}(B)^{1/p'} \varphi(\widetilde{B})^{1/p}.
\end{equation*}

On the other hand, 
\begin{equation*}
%\aligned
\left|\int_B f\right|\le \left(\int_B f^p w\right)^{1/p} w^{1-p'}(B)^{1/p'} 
\le  \varphi(\widetilde{B})^{1/p} \|f\|_{\mathcal {LM}^{p}(\varphi,w)}
w^{1-p'}(B)^{1/p'}.
%\endaligned
\end{equation*}
Taking the supremum over the functions $f$ with $\|f\|_{\mathcal {LM}^{p}(\varphi,w)}\le 1$ proves the claim.

Let $p=1$. We have \eqref{locchar}. Moreover, if $w\in A_{1,\text{loc}}$, 
\begin{equation*}
\left|\int_B f\right| \lesssim \int_B |f|w\  \frac{|B|}{w(B)}\le \varphi(\widetilde{B})\frac{|B|}{w(B)}\|f\|_{\mathcal {LM}^{1}(\varphi,w)}.
 \end{equation*}
 From here we obtain
 \begin{equation*}
\|\chi_B\|_{\mathcal {LM}^{1}(\varphi,w)'}\lesssim \varphi(\widetilde{B})\frac{|B|}{w(B)},
 \end{equation*}
 which together with \eqref{locchar} implies that $w\in A(\mathcal {LM}^{1}(\varphi))$.
 
 Let $E$ be a subset of $B$ of positive measure. Then
 \begin{equation*}
\|\chi_B\|_{\mathcal {LM}^{1}(\varphi,w)'}\ge \frac{1}{\|\chi_E\|_{\mathcal {LM}^{1}(\varphi,w)}}\int_B \chi_E \sim \frac{\varphi(\widetilde{B})|E|}{w(E)}. 
\end{equation*}
For $w\in  A(\mathcal {LM}^{1}(\varphi))$, using \eqref{locchar} we obtain 
\begin{equation}\label{fora1}
\frac{w(B)}{|B|}\lesssim \frac{w(E)}{|E|}.
\end{equation}
 With $E=\{x\in B: w(x)\le \inf_B w+\epsilon\}$ we obtain
\begin{equation*}
\frac{w(B)}{|B|}\lesssim \inf_B w+\epsilon.
\end{equation*}
Since this holds for any $\epsilon>0$ we get the $A_{1,\text{loc}}$ condition. 

The comparability of the constants $[w]_{A_{\textrm{loc}}(\mathcal {LM}^{p}(\varphi))}$ and  $[w]_{A_{p,\text{loc}}}$ is a consequence of the proof.
\end{proof}
 
\begin{proposition}\label{maxlocmor}
(a) Let $1<p<\infty$. Then $M_{1/4,\text{loc}}$ is bounded on $\mathcal {LM}^{p}(\varphi,w)$ if and only if $w\in A_{p,\text{loc}}$.

(b) Let $1\le p<\infty$. Then $M_{1/4,\text{loc}}$ is bounded from $\mathcal {LM}^{p}(\varphi,w)$ to $W\mathcal {LM}^{p}(\varphi,w)$ if and only if $w\in A_{p,\text{loc}}$.

In both cases, equivalently, if $A_{\textrm{loc}}(\mathcal {LM}^{p}(\varphi))$.  Moreover, in case (b) the operator norm is equivalent to $[w]_{A_{\textrm{loc}}(\mathcal {LM}^{p}(\varphi))}$.
\end{proposition}

\begin{proof}
The necessity is proved as in Theorem \ref{necbasis}.

For the sufficiency in the case $1<p<\infty$ it is enough to prove the case (a). Take a ball $B$ such that $|c_B|=4r_B$. By the definition of the local maximal operator $M_{1/4,\text{loc}}$ we have that when we evaluate $M_{1/4,\text{loc}}f(y)$ only the values of $f$ on balls of the form $B(z,|z|/4)$ containing $y$ are involved. One can check that those balls are inside $5B$. Hence, for $y\in B$ we have
\begin{equation*}
M_{1/4,\text{loc}} f(y)=M_{1/4,\text{loc}} (f\chi_{5B})(y).
\end{equation*}
Then
\begin{equation*}
\aligned
\frac 1{\varphi(B)}\int_{B}|M_{1/4,\text{loc}}f|^p w &= \frac 1{\varphi(B)}\int_{B}|M_{1/4,\text{loc}}(f\chi_{5B})|^p w \\
&\lesssim \frac 1{\varphi(B)}\int_{\widetilde{5B}}|f|^p w\le  \frac {\varphi(\widetilde{5B})}{\varphi(B)} \|f\|_{\mathcal {LM}^{p}(\varphi,w)}^p.
\endaligned
\end{equation*}
The result follows from the doubling property of $\varphi$.

The sufficiency for $p=1$ in (b) is proved in the same way, but we use the weak-type (1,1) inequality for $M_{1/4,\text{loc}}$ and  $A_{1,\text{loc}}$ instead of the strong type.

To prove that the operator norm in the weak case is equivalent to $[w]_{A_{\textrm{loc}}(\mathcal {LM}^{p}(\varphi))}$, we obtain the lower bound from  Theorem \ref{necbasis}. The upper bound is obtained taking into account that  $M_{1/4,\text{loc}}$ is bounded from $L^p(w)$ to $L^{p,\infty}(w)$ with constant comparable to $[w]_{A_{p,\text{loc}}}$. For $p=1$ this is implicit in the proof of \cite[Theorem 1.1]{HSV14}, and for $p>1$ it is enough to adapt such proof.
\end{proof}

For the maximal operator $M_0$ we have the following result.

\begin{proposition}\label{maxloccero}
 Let $1\le p<\infty$. Then $M_{0}$ is bounded on $\mathcal {LM}^{p}(\varphi,w)$ if and only if $w\in A_0(\mathcal {LM}^{p}(\varphi))$. Equivalently, if and only if $M_{0}$ is bounded from $\mathcal {LM}^{p}(\varphi,w)$ to $W\mathcal {LM}^{p}(\varphi,w)$. Moreover, the operator norm in both cases is equivalent to $[w]_{A_{0}(\mathcal {LM}^{p}(\varphi))}$.
\end{proposition}

The proof is the same as the one given in Theorem \ref{teocero}.

Using the last two propositions and \eqref{maxdec}, Theorem \ref{mainthmloc} is proved.

It is immediate to check that Theorem \ref{morcal} holds for $\mathcal {LM}^{p}(\varphi,w)$, and that Remark \ref{rem42} and the comments following it are also valid in the setting of weighted local Morrey spaces.

%%%%%%%%%%%%%%%%%%%%%%%%%%%%%
\section{Morrey boundedness from extrapolation of Lebesgue estimates}\label{sei}

In our previous papers (\cite{DR18, DR19, DR20}) we extended the weighted Lebesgue estimates appearing in the extrapolation theorem to the Morrey setting. In the framework of the current paper the corresponding result is the following. 

\begin{theorem}\label{teo1K}
Let $1\le p_0<\infty$ and let $\mathcal F$ be a collection of nonnegative measurable pairs of functions. Assume that for every $(f,g)\in \mathcal F$ and every $v\in A_{p_0}$ we have
\begin{equation}\label{hypextra}
\|g\|_{L^{p_0}(v)}\lesssim \|f\|_{L^{p_0}(v)},
\end{equation}
where the constant involved in the inequality does not depend on the pair $(f,g)$ and it depends on $v$ only in terms of $[v]_{A_{p_0}}$. 
Then for $1<p<\infty$ and $w\in A_{p, \textrm{loc}}$ satisfying
\begin{equation}\label{extrasuf}
\sup_B \frac 1{|B|} \|\chi _B\|_{\mathcal M^{q}(\varphi,w)}  \|M(\chi_B)^{\frac {1}s}\|_{\mathcal M^{q}(\varphi,w)'} <\infty,
\end{equation}
for some $q<p$ and some $s>1$,
 it holds
 \begin{equation}\label{boundmorrey}
\|g\|_{\mathcal M^{p}(\varphi,w)}\lesssim  \|f\|_{\mathcal M^{p}(\varphi,w)}.
\end{equation}

If \eqref{hypextra} holds for $p_0=1$, then $q=p$ can be taken in \eqref{extrasuf} and  moreover the conclusion \eqref{boundmorrey} is valid also for $p=1$.
\end{theorem}

\begin{remark}
Before proceeding with the proof of the theorem we make the following observations. Assuming that \eqref{extrasuf} holds for fixed $q<p$ and $s>1$, we have: (i) $s$ can be replaced by any  other value in $(1,s)$; (ii) $q$ can be replaced by any other value in $(q,p)$ (changing the value of $s$ if needed).

Statement (i) holds trivially because $M(\chi_B)\le 1$; hence, $M(\chi_B)^{\frac {1}s}$ is increasing with $s$. To prove (ii) take $\tilde q>q$ and denote $r=\tilde{q}/q$. On the one hand, 
\begin{equation*}
\|\chi _B\|_{\mathcal M^{q}(\varphi,w)} = \|\chi _B\|_{\mathcal M^{\tilde{q}}(\varphi,w)}^r, 
\end{equation*}  
using \eqref{normscale}. On the other hand, for some $s_1>1$ to be precised, using H\"older's inequality,
\begin{equation*}
\int M(\chi_B)^{\frac {1}{s_1}} f \le \left(\int M(\chi_B)^{\frac {1}s} f^r\right)^{\frac{1}{r}} \left(\int M(\chi_B)^{\left(\frac {1}{s_1}-\frac 1{sr}\right) r'}\right)^{\frac{1}{r'}}.
\end{equation*}
We choose $s_1>1$ such that the exponent of $M(\chi_B)$ in the last integral is greater than $1$. This is equivalent to requiring that
\begin{equation*}
\frac {1}{s_1}>\frac 1{sr}+\frac 1{r'}.
\end{equation*}
Since the right-hand side is smaller than $1$, it is possible to choose $s_1>1$ satisfying the condition.
Now we can use the boundedness of $M$ to write
\begin{equation*}
\aligned
\int M(\chi_B)^{\frac {1}{s_1}} f &\lesssim \|M(\chi_B)^{\frac {1}s}\|_{\mathcal M^{q}(\varphi,w)'}^{\frac{1}{r}} \|f^r\|_{\mathcal M^{q}(\varphi,w)}^{\frac{1}{r}} |B|^{\frac{1}{r'}}\\
&= \|M(\chi_B)^{\frac {1}s}\|_{\mathcal M^{q}(\varphi,w)'}^{\frac{1}{r}} \|f\|_{\mathcal M^{\tilde q}(\varphi,w)} |B|^{\frac{1}{r'}}.
\endaligned
\end{equation*}
Taking the supremum on the functions $f$ such that $\|f\|_{\mathcal M^{\tilde q}(\varphi,w)}\le 1$ we get
\begin{equation*}
\|M(\chi_B)^{\frac {1}{s_1}}\|_{\mathcal M^{\tilde q}(\varphi,w)'}\lesssim \|M(\chi_B)^{\frac {1}s}\|_{\mathcal M^{q}(\varphi,w)'}^{\frac{1}{r}} |B|^{\frac{1}{r'}}.
\end{equation*}
This gives \eqref{extrasuf} for $\tilde q$, with $s_1$ instead of $s$.  
\end{remark}

\begin{proof}[Proof of Theorem \ref{teo1K}]
 \textit{Case $p_0=1$ and $1<p<\infty$}. The assumption is now
\begin{equation}\label{hypbat}
\|g\|_{L^{1}(v)}\lesssim \|f\|_{L^{1}(v)},\quad\text{for } v\in A_1.
\end{equation}

 Let $p>1$. Since $w\in A_{p, \textrm{loc}}$, we know that $w^{1-p'}\in A_{p', \textrm{loc}}$ and we can choose $s>1$ such that $w^{1-p'}\in A_{p'/s, \textrm{loc}}$. 
 
 Let $B$ be a ball with $|c_B|>4r_B$. There exists nonnegative $h\in L^{p'}(w, B)$ with norm $1$ such that 
\begin{equation*}
\int_B h^{p'}w=1 \quad \text{and} \quad \left(\int_{B} g^p w\right)^{\frac 1p}=\int_{B} g h w.
\end{equation*}
Since $hw\chi_B\le M(h^sw^s\chi_B)^{1/s}$ and $M(h^sw^s\chi_B)^{1/s}$ is an $A_1$ weight for $s>1$, we can write
\begin{equation*}
\int_{B} g h w\le \int_{\rn} g M(h^sw^s\chi_B)^{1/s}\le C \int_{\rn} f M(h^sw^s\chi_B)^{1/s},
\end{equation*}
where in the second inequality we use \eqref{hypbat}. To ensure that $M(h^sw^s\chi_B)^{1/s}$ is in $A_1$ we need to check that it is finite almost everywhere and for this it suffices to show that $h^sw^s\chi_B$ is integrable. We prove this and get a bound for future use. We have
\begin{align} 
\begin{split} \label{p3}
	 \left(\int_{B} h^s w^{s-1} w\right)^\frac{1}{s}
	&\le\left(\int_{B} h^{p'} w\right)^\frac{1}{p'}
	\left(\int_{B} w^{\frac{s(p'-1)}{p'-s}} \right)^{\frac1{s}-\frac{1}{p'}}\\ 
	& \lesssim\  |B|^{\frac{1}{s}}\left(\int_{B} w^{1- p'}\right)^{-\frac{1}{p'}}
		\lesssim\left(\int_{B} w\right)^{\frac{1}{p}}|B|^{-\frac{1}{s'}},
\end{split} 
\end{align}
where the second inequality holds because $w^{1-p'}\in A_{p'/s, \textrm{loc}}$ (the exponent of $w$ in the integral can be written as $(1-p')/(1-(p'/s)$) and in the last one we use 
\begin{equation*}
|B|\le \left(\int_{B} w\right)^{\frac{1}{p}} \left(\int_{B} w^{1- p'}\right)^{\frac{1}{p'}}.
\end{equation*}  

We decompose the integral of $f M(h^sw^s\chi_B)^{1/s}$ over $\rn$  in the form
\begin{equation}\label{bitan}
 \int_{2B} f M(h^sw^s\chi_B)^{1/s} + \int_{\rn \setminus 2B} f M(h^sw^s\chi_B)^{1/s}.
\end{equation}

To deal with the first term in \eqref{bitan} we use H\"older's inequality,
\begin{equation*}
\int_{2B} f M(h^sw^s\chi_B)^{1/s}\le \left(\int_{2B} f^p w\right)^{\frac 1p} \left(\int_{2B} M(h^sw^s\chi_B)^{p'/s} w^{1-p'}\right)^{\frac 1{p'}}.
\end{equation*}
The function $h^sw^s\chi_B$ is supported on $B$ and for $y\in 2B$ the value of $M(h^sw^s\chi_B)(y)$ only needs to take into account balls of radii less than $3r_B$. Therefore, 
\begin{equation}\label{toloc}
M(h^sw^s\chi_B)(y)\sim M_{\textrm{loc}}(h^sw^s\chi_B)(y)+\frac 1{|B|}\int_B h^sw^s.
\end{equation}
Since $s$ has been chosen such that $w^{1-p'}\in A_{p'/s, \textrm{loc}}$, we use the weighted boundedness of $M_{\textrm{loc}}$ for the first summand and \eqref{p3} for the second to see that the last term is less than a constant. 
On the other hand,
\begin{equation}\label{bib}
 \left(\int_{2B} f^p w\right)^{\frac 1p} \lesssim   \varphi(2B)^{\frac 1p}  
\|f\|_{\mathcal M^{p}(\varphi,w)}.
\end{equation}
From the doubling property of $\varphi$ we obtain the desired bound. 

To deal with the integral on $\rn \setminus 2B$ we first realize that for a function supported on $B$ the maximal function at a point $y$ outside $2B$ only needs to take into account averages on balls containing $y$ and intersecting $B$. Therefore, we  are only concerned with balls whose radii are comparable to $|y-c_B|$. Thus for $y\in \rn \setminus 2B$ we have
\begin{equation*}
\aligned
& M(h^sw^s\chi_B)(y)^{\frac 1s} \sim \left(\frac 1{|y-c_B|^n} \int_B h^sw^s\right)^{\frac 1s}
\sim \left(\frac {M(\chi_B)(y)}{|B|} \int_B h^sw^s\right)^{\frac 1s}
\\ & \qquad\qquad \lesssim M(\chi_B)(y)^{\frac {1}s} \frac {w(B)^{\frac{1}{p}}}{|B|}
\lesssim  \frac {M(\chi_B)(y)^{\frac {1}s}}{|B|}\,\varphi(B)^{\frac{1}{p}}\|\chi _B\|_{\mathcal M^{p}(\varphi,w)},
\endaligned
\end{equation*}
using \eqref{p3} in the third step. Then
\begin{equation*}
\aligned
&\int_{\rn \setminus 2B} f   M(h^sw^s\chi_B)^{1/s}\lesssim \left(\int_{\rn \setminus 2B} f M(\chi_B)^{\frac {1}s}\right) \frac {\varphi(B)^{\frac{1}{p}}}{|B|}\,\|\chi _B\|_{\mathcal M^{p}(\varphi,w)}\\
&\qquad \lesssim \|f\|_{\mathcal M^{p}(\varphi,w)} \|M(\chi_B)^{\frac {1}s}\|_{\mathcal M^{p}(\varphi,w)'}\frac {\varphi(B)^{\frac{1}{p}}}{|B|}\,\|\chi _B\|_{\mathcal M^{p}(\varphi,w)}.
\endaligned
\end{equation*}
The desired estimate follows from \eqref{extrasuf} with $q=p$.

 \textit{Case $p_0=1$ and $p=1$}. Let $B$ be a ball with $|c_B|>4r_B$ and $w\in A_{1, \textrm{loc}}$ satisfying \eqref{extrasuf} with $q=1$. Choose $s>1$ such that $w^s\in A_{1, \textrm{loc}}$. (This is the same as satisfying a reverse H\"older condition with exponent $s$ over the balls used in the definition of $M_{\textrm{loc}}$.) Then
\begin{equation*}
\int_{B} g w\le \int_{\rn} g M(w^s\chi_B)^{1/s}\le C \int_{\rn} f M(w^s\chi_B)^{1/s},
\end{equation*} 
using \eqref{hypbat} in the second inequality. We are like in \eqref{bitan} with $h\equiv 1$. Over $2B$ we have \eqref{toloc} and since $w^s\in A_{1, \textrm{loc}}$ we deduce that $M_{\text{loc}}(w^s\chi_B)^{1/s}(y)\lesssim w(y)$ a.e. and we are like in \eqref{bib} with $p=1$. 

For $y\in \rn \setminus 2B$ we have
\begin{equation*}
\aligned
M(w^s\chi_B)(y)^{\frac 1s} &\sim \left(\frac 1{|y-c_B|^n} \int_B w^s\right)^{\frac 1s}\lesssim \frac {M(\chi_B)(y)^{\frac {1}s}}{|B|}\int_{B} w\\
&\lesssim  \frac {M(\chi_B)(y)^{\frac {1}s}}{|B|}\,\varphi(B)\|\chi _B\|_{\mathcal M^{1}(\varphi,w)},
\endaligned
\end{equation*}
where the second inequality is a consequence of $w^s\in A_{1, \textrm{loc}}$. We are in the same situation as in the previous part of the proof. 

 \textit{Case $p_0>1$ and $p>1$}. Let $w\in A_{p, \textrm{loc}}$. Take $q\in (1,p)$ for which \eqref{extrasuf} holds and $w\in A_{q, \textrm{loc}}$. By the extrapolation theory for weighted Lebesgue spaces we can take any $p_0\in (1,\infty)$ in \eqref{hypextra}. In particular,
\begin{equation}\label{hyp1}
\|g^{p/q}\|_{L^{1}(v)}\lesssim \|f^{p/q}\|_{L^{1}(v)},\quad\text{for } v\in A_1.
\end{equation}  
Using the first part of the proof we deduce
\begin{equation*}
\|g^{p/q}\|_{\mathcal M^{q}(\varphi,w)}\lesssim \|f^{p/q}\|_{\mathcal M^{q}(\varphi,w)},
\end{equation*}  
which is equivalent to \eqref{boundmorrey}.
\end{proof}

A comment similar to Remark \ref{rem42} can be made here: there are weights satisfying \eqref{apmdef} which do not satisfy \eqref{extrasuf} for $q\le p$. Indeed, again in the case of the Samko type space the Hardy-Littlewood maximal operator is bounded for $w(x)=|x|^{\lambda-n}$, but some singular  integrals are not (for instance, the Hilbert transform in the case $n=1$, see \cite{Sam09, DR19}). This means that \eqref{apmdef} is not sufficient and legitimates the presence of the stronger condition \eqref{extrasuf} in Theorem \ref{teo1K}. 

The main result of \cite{Ru18} can be applied here: the boundedness of the Hardy-Littlewood maximal operator on both $X$ and $X'$ is equivalent to the boundedness of the Riesz transforms (the Hilbert transform if we are in $\R$) on  $X$,  in the sense that the norm inequality for the latter holds for functions in $X\cap L^2$. This is proved for Banach lattices $X$ with a Fatou property different from the one mentioned in Subsection \ref{kothesub}, namely, if $f_n\to f$ a.e. and $\|f_n\|_X\le 1$, then $f\in X$ and $\|f\|_X\le 1$. If $X=\mathcal M^{p}(\varphi,w)$ this property is fulfilled (use Fatou's lemma) and the result can be applied.  Hence, if $M$ is bounded on $\mathcal M^{p}(\varphi,w)$ but the Riesz transforms are not, then $M$ is not bounded on $\mathcal M^{p}(\varphi,w)'$.

We can obtain an extrapolation-type result also in the framework of weighted local Morrey spaces. This is the counterpart of Theorem \ref{teo1K}.

\begin{theorem}\label{teo1Kloc}
Let $1\le p_0<\infty$ and let $\mathcal F$ be a collection of nonnegative measurable pairs of functions. Assume that for every $(f,g)\in \mathcal F$ and every $v\in A_{p_0}$ we have
\begin{equation}\label{hypextra2}
\|g\|_{L^{p_0}(v)}\lesssim \|f\|_{L^{p_0}(v)},
\end{equation}
where the constant involved in the inequality does not depend on the pair $(f,g)$ and it depends on $v$ only in terms of $[v]_{A_{p_0}}$. 
Then for $1<p<\infty$ and $w\in A_{p, \textrm{loc}}$ satisfying
\begin{equation}\label{extrasuf2}
\sup_B \frac 1{|B|} \|\chi _B\|_{\mathcal {LM}^{q}(\varphi,w)}  \|M(\chi_B)^{\frac {1}s}\|_{\mathcal {LM}^{q}(\varphi,w)'} <\infty,
\end{equation}
for some $q<p$ and some $s>1$, and all balls $B$ centered at the origin,
 it holds
 \begin{equation}\label{boundmorrey2}
\|g\|_{\mathcal {LM}^{p}(\varphi,w)}\lesssim  \|f\|_{\mathcal {LM}^{p}(\varphi,w)}.
\end{equation}
\end{theorem}

\begin{proof}
The proof is almost the same as the one given for Theorem \ref{teo1K}. We only mention the small changes. The ball $B$ is such that $|c_B|=4r_B$, and $\widetilde B$ is as before. 

The proof runs as before until \eqref{bib}. Then  
\begin{equation*}
 \left(\int_{2B} f^p w\right)^{\frac 1p}\le \left(\int_{\widetilde {2B}} f^p w\right)^{\frac 1p} \lesssim   \varphi(\widetilde {2B})^{\frac 1p}  
\|f\|_{\mathcal {LM}^{p}(\varphi,w)}.
\end{equation*}

When we are in $\rn \setminus 2B$ we have 
\begin{equation*}
 M(h^sw^s\chi_B)(y)^{\frac 1s} \lesssim M(\chi_B)(y)^{\frac {1}s} \frac {w(B)^{\frac{1}{p}}}{|B|}
\lesssim  M(\chi_{\widetilde B})(y)^{\frac {1}s}  \frac {w(\widetilde B)^{\frac{1}{p}}}{|\widetilde B|},
\end{equation*}
where the first inequality is in the proof of Theorem \ref{teo1K} and the second one is immediate because $B\subset \widetilde B$ and both have comparable size. From here the proof runs as before.
\end{proof}

We can apply again the result in \cite{Ru18}: the Riesz transforms are bounded on $\mathcal {LM}^{q}(\varphi,w)\cap L^2$ if and only the Hardy-Littlewood maximal operator is bounded on $\mathcal {LM}^{q}(\varphi,w)$ and $\mathcal {LM}^{q}(\varphi,w)'$. In particular, $w\in  A(\mathcal {LM}^{p}(\varphi))$ is necessary for the boundedness of the Riesz transforms (compare with \cite[Theorem 1.6]{NST19}). 

There are many applications of the theorems in this section to operators $T$ satisfying \eqref{hypextra} for pairs $(|f|, |Tf|)$. We refer to the examples presented in \cite{DR18}, for example. From the proof of the theorems it is also deduced the embedding of the  Morrey space into the union of weighted Lebesgue spaces. This allows us to define the operator on the Morrey space by restriction. See Section 4 of \cite{DR20}, for instance.   

The applications to boundedness on weighted local Morrey spaces are new.

%%%%%%%%%%%%%%%%%%%%%%%%%%%%%
\section{Relation with $A_p$ weights and with previous results}\label{zazpi}

In this section we compare the results of this paper with our results in previous papers in the case 
\begin{equation}\label{embed}
\varphi (B)= r_B^{\lambda_1} w(B)^{\lambda_2/n}.
\end{equation}
We assume that $0<\lambda_1+\lambda_2<n$ (see Proposition 2.4 in \cite{DR20}). We also assume that $0\le \lambda_2\le n$.  

\subsection{The case of weighted global Morrey spaces}

First we obtain a necessary condition, a quantitative reverse doubling property, which apparently has been overlooked in the literature concerning the Samko-type weighted Morrey spaces. Another necessary condition saying that the weights are in a certain Muckenhoupt class $A_q$ with $q$ depending on $p,\lambda_1$ and $\lambda_2$ was also obtained in \cite{DR20}. We will use the following embedding result from \cite{DR20}.

\begin{lemma}\cite[Lemma 2.5]{DR20}\label{emblema}
Let $1\le p<\infty$, $0\le \lambda_1, \lambda_2 <n$ and $\lambda_1 + \lambda_2 < n$.  The embedding 
\[
L^{\frac {pn}{n-\lambda_1-\lambda_2}}(w^{\frac{n-\lambda_2}{n-\lambda_1-\lambda_2}})\hookrightarrow \mathcal M^{p}(\varphi,w)
\] 
holds for $\varphi$ as in \eqref{embed} and a constant depending only on $n$, $\lambda$ and $p$, not on $w$.
\end{lemma}

\begin{proposition}\label{rhprop}
Let $1\le p<\infty$, $0\le \lambda_1, \lambda_2 <n$ and $0<\lambda_1 + \lambda_2 < n$. Let $w\in A(\mathcal M^{p}(\varphi))$. 

(a) If $B_1$ and $B_2$ are balls such that $B_1\subset B_2$, then it holds that
\begin{equation}\label{rh2}
\frac {w(B_1)}{w(B_2)}\lesssim \left( \frac {|B_1|}{|B_2|} \right)^{\frac {\lambda_1}{n-\lambda_2}}.
\end{equation}
In other words, $w\in RD_{\frac {\lambda_1}{n-\lambda_2}}$.

(b) The weight $w$ is in $A_{\frac{np+\lambda_1}{n-\lambda_2}}$.
\end{proposition}

\begin{proof}
(a) From the definition of the norm we have
 \begin{equation}\label{eq73}
\|\chi _{B_2}\|_{\mathcal M^{p}(\varphi,w)} \ge \left( \frac {w(B_1)^{1-\frac{\lambda_2}n}}{|B_1|^{\frac{\lambda_1}n} }\right)^{\frac 1p}.
\end{equation}
On the other hand, for $\alpha$ to be chosen later we write
 \begin{equation*}
\|\chi _{B_2}\|_{\mathcal M^{p}(\varphi,w)'} \ge \frac {\int_{B_2}w^\alpha}{\|w^\alpha\chi _{B_2}\|_{\mathcal M^{p}(\varphi,w)}}\gtrsim \frac {\int_{B_2}w^\alpha}{\|w^\alpha\chi _{B_2}\|_{L^{\frac {pn}{n-\lambda_1-\lambda_2}}(w^{\frac{n-\lambda_2}{n-\lambda_1-\lambda_2}})}},
\end{equation*}
using Lemma \ref{emblema}. We choose $\alpha$ such that
\begin{equation*}
\alpha = \alpha\, \frac {pn}{n-\lambda_1-\lambda_2} + \frac{n-\lambda_2}{n-\lambda_1-\lambda_2},
\end{equation*}
that is $\alpha = \frac{\lambda_2-n}{n(p-1)+\lambda_1+\lambda_2}$. Then
 \begin{equation}\label{eq74}
\|\chi _{B_2}\|_{\mathcal M^{p}(\varphi,w)'} \ge \left(\int_{B_2}w^\alpha\right)^{\frac{n(p-1)+\lambda_1+\lambda_2}{np}}.
\end{equation}
Use now the inequality
\begin{equation*}
|B_2|\le \left(\int_{B_2}w\right)^{\frac 1q} \left(\int_{B_2}w^{\frac 1{1-q}}\right)^{\frac 1{q'}}
\end{equation*}
with $1/(1-q)=\alpha$. This gives
\begin{equation*}
\left(\int_{B_2}w^\alpha\right)^{\frac{n(p-1)+\lambda_1+\lambda_2}{np+\lambda_1}} \left(\int_{B_2}w\right)^{\frac{n-\lambda_2}{np+\lambda_1}}\ge |B_2|,
\end{equation*}
which implies
\begin{equation*}
\|\chi _{B_2}\|_{\mathcal M^{p}(\varphi,w)'} \ge \left(\int_{B_2}w\right)^{\frac{\lambda_2-n}{np}}|B_2|^{\frac{np+\lambda_1}{np}}.
\end{equation*}
Therefore,
\begin{equation*}
\frac {\|\chi _{B_2}\|_{\mathcal M^{p}(\varphi,w)} \|\chi _{B_2}\|_{\mathcal M^{p}(\varphi,w)'}}{|B_2|} \ge 
\left( \frac {w(B_1)^{1-\frac{\lambda_2}n}}{|B_1|^{\frac{\lambda_1}n} } \frac {|B_2|^{\frac{\lambda_1}n} } {w(B_2)^{1-\frac{\lambda_2}n}} 
\right)^{\frac 1p}.
\end{equation*}
We get \eqref{rh2} from this inequality and the $A(\mathcal M^{p}(\varphi))$ condition.

(b) For a ball $B$ we use \eqref{eq73} with $B_1=B_2=B$ and \eqref{eq74} with $B_2=B$. Then
\begin{equation*}
\aligned
&\frac {\|\chi _{B}\|_{\mathcal M^{p}(\varphi,w)} \|\chi _{B}\|_{\mathcal M^{p}(\varphi,w)'}}{|B|} \\
&\qquad\ge \frac 1{|B|} \left( \frac {w(B)^{1-\frac{\lambda_2}n}}{|B|^{\frac{\lambda_1}n} }\right)^{\frac 1p}\left(\int_{B}w^{\frac{\lambda_2-n}{n(p-1)+\lambda_1+\lambda_2}}\right)^{\frac{n(p-1)+\lambda_1+\lambda_2}{np}}\\
&\qquad =  \left( \frac {w(B)^{\frac{n-\lambda_2}{np+\lambda_1}}\ w^{\frac{\lambda_2-n}{n(p-1)+\lambda_1+\lambda_2}}(B)^{\frac{n(p-1)+\lambda_1+\lambda_2}{np+\lambda_1}}}{|B|}\right)^{\frac {np+\lambda_1}{np}}.
\endaligned
\end{equation*}
Therefore, if $w\in A(\mathcal M^{p}(\varphi))$ then $w\in A_{\frac{np+\lambda_1}{n-\lambda_2}}$.
\end{proof}

In the case $\lambda_2=0$ and $\lambda_1\in (0,n)$ we are in the Samko-type space and condition \eqref{rh2} is written in the form \eqref{rhsamko}, that is, $w\in RD_{\frac {\lambda_1}n}$. 

 In our previous papers \cite{DR18, DR19, DR20} we imposed a reverse H\"older condition to control the quotients of the form  ${w(B_1)}/{w(B_2)}$, using \eqref{ineqRH}. Actually, a reverse doubling condition would be enough. In \cite{DR19} we even wondered about the necessity of a reverse H\"older condition with exponent $\sigma=n/(n-\lambda)$, for which \eqref{ineqRH} would imply \eqref{rhsamko} (see the comments at the end of \cite[Section 5]{DR19}). With the current result we see that the weaker condition $w\in RD_{\frac {\lambda_1}n}$ is necessary.  
 
From Proposition \ref{rhprop} we deduce in particular that $w$ is doubling. For other values of $\lambda_1$ and $\lambda_2$ not considered in the proposition we assume that $w$ is doubling. This ensures that $\varphi$ is doubling. We also have that $w$ is reverse doubling and if $w\in RD_\delta$ we have for balls $B_1\subset B_2$,
\begin{equation*}
\frac{\varphi(B_1)}{\varphi(B_2)}=\left(\frac{r_{B_1}}{r_{B_2}}\right)^{\lambda_1} \left(\frac{w(B_1)}{w(B_2)}\right)^{\lambda_2/n}\lesssim \left(\frac{r_{B_1}}{r_{B_2}}\right)^{\lambda_1+\delta \lambda_2}.
\end{equation*}
Then we impose the condition $\lambda_1+\delta \lambda_2>0$, which implies that $\varphi$ is reverse doubling. This condition is weaker than the condition  $\lambda_1\sigma_w'+\lambda_2>0$ imposed in \cite{DR20}, due to \eqref{ineqRH}. In particular, we do not assume a priori that $w$ satisfies a reverse H\"older condition. Notice that $\lambda_1$ can be negative (in particular, it will be always negative for $\lambda_2=n$).

\begin{remark}\label{rm73}
 When $w$ satisfies \eqref{rh2}, we claim that the norm of the characteristic function of the ball $B$ in $\mathcal M^{p}(\varphi,w)$ is comparable to 
 \begin{equation}\label{eqrm73}
\left(\frac {w(B)^{1-\lambda_2/n}}{|B|^{\lambda_1/n}}\right)^{1/p}.
\end{equation}
Indeed, according to Remark \ref{rm22} we only need to consider in \eqref{normdef} balls contained in $2B$. For such a ball $Q$, we have
\begin{equation*}
\frac {w(Q\cap B)}{|Q|^{\lambda_1/n}w(Q)^{\lambda_2/n}}\le \frac {w(Q)^{1-\lambda_2/n}}{|Q|^{\lambda_1/n}}\lesssim \frac {w(2B)^{1-\lambda_2/n}}{|2B|^{\lambda_1/n}}\lesssim \frac {w(B)^{1-\lambda_2/n}}{|B|^{\lambda_1/n}},
\end{equation*}
which implies the claim. In particular, this shows that characteristic functions of balls are in the Morrey space. Notice also that if the norm of $\chi_B$ in $\mathcal M^{p}(\varphi,w)$ is comparable to \eqref{eqrm73}, then \eqref{rh2} holds.

In the case $0<\lambda_2 \le n$ and $0>\lambda_1>-\delta \lambda_2$, \eqref{eqrm73} for $\|\chi_B\|_{\mathcal M^{p}(\varphi,w)}$ is immediate without further assumption on $w$. 
\end{remark}

In the next result we prove that $w\in A_p$ together with condition \eqref{rh2}, when it is needed, imply that the Muckenhoupt-Morrey condition holds. This could be deduced from  boundedness results for the Hardy-Littlewood maximal operator in \cite{DR20} and Theorem \ref{necbasis}, but we give here a direct proof. In this way, the result for $M$ in  \cite{DR20} is obtained as a consequence of Theorem \ref{teohl}.      

\begin{proposition}\label{apsuf}
Let $1\le p<\infty$, $0\le \lambda_1, \lambda_2 <n$, $0<\lambda_1 + \lambda_2 < n$. Then if $w\in A_p$ and $w\in RD_{\frac {\lambda_1}{n-\lambda_2}}$ (that is, $w$ satisfies \eqref{rh2}), then $w\in A(\mathcal M^{p}(\varphi))$. 
Furthermore, if $w\in RD_{\delta}$ with $\delta>\lambda_1/(n-\lambda_2)$, then $w\in A_p$ satisfies the stronger condition \eqref{extrasuf} for $q=p$, and also for some $1<q<p$ if $p>1$. 

If $w\in A_p\cap RD_{\delta}$, $0<\lambda_2 \le n$, and $0>\lambda_1>-\delta \lambda_2$, it holds \eqref{extrasuf} and in particular $w\in A(\mathcal M^{p}(\varphi))$.
\end{proposition}

\begin{proof}
First we use Remark \ref{rm73} to write a bound for $\|\chi _{B}\|_{\mathcal M^{p}(\varphi,w)}$, namely,\begin{equation}\label{norma1}
\|\chi _{B}\|_{\mathcal M^{p}(\varphi,w)}\lesssim \left(\frac {w(B)^{1-\lambda_2/n}}{|B|^{\lambda_1/n}}\right)^{1/p}.
\end{equation}

The bound for $\|\chi _{B}\|_{\mathcal M^{p}(\varphi,w)'}$ is as follows. If $\|f\|_{\mathcal M^{p}(\varphi,w)}\le 1$ and $p>1$ we have 
\begin{equation*}
\int_B f\le \left(\int_B f^p w\right)^{1/p} w^{1-p'}(B)^{1/p'}\le [|B|^{\lambda_1/n}w(B)^{\lambda_2/n}]^{1/p}w^{1-p'}(B)^{1/p'},
\end{equation*}
and the right-hand side is a bound for $\|\chi _{B}\|_{\mathcal M^{p}(\varphi,w)'}$. If $p=1$, 
\begin{equation*}
\int_B f\lesssim \left(\int_B f w\right) \frac{|B|}{w(B)} \le |B|^{1+\lambda_1/n}w(B)^{-1+\lambda_2/n},
\end{equation*}
and we obtain a bound for $\|\chi _{B}\|_{\mathcal M^{1}(\varphi,w)'}$.

Inserting the obtained bounds in the definition of $A(\mathcal M^{p}(\varphi))$ we see that the inclusion of the statement holds. 

Assume now that $w\in RD_{\delta}$ with $\delta>\lambda_1/(n-\lambda_2)$. To check \eqref{extrasuf} we need to get a bound for  $\|M(\chi _{B})^{1/s}\|_{\mathcal M^{p}(\varphi,w)'}$. Using the pointwise bound 
\begin{equation}\label{pb74}
M(\chi _{B})^{1/s}(x)\lesssim \sum_{k=1}^\infty 2^{-kn/s} \chi _{2^kB}(x)
\end{equation}
and the previous calculations, we have for $p>1$  
\begin{equation*}
\aligned
\|M(\chi _{B})^{1/s}\|_{\mathcal M^{p}(\varphi,w)'} &\lesssim \sum_{k=1}^\infty 2^{-kn/s} \|\chi _{2^kB}\|_{\mathcal M^{p}(\varphi,w)'}\\
&\lesssim \sum_{k=1}^\infty 2^{-kn/s} [|2^kB|^{\lambda_1/n}w(2^kB)^{\lambda_2/n}]^{1/p}w^{1-p'}(2^kB)^{1/p'}.
\endaligned
\end{equation*}
Since $w\in A_p$ we have 
\begin{equation*}
w^{1-p'}(2^kB)^{1/p'}\lesssim |2^k B| w(2^kB)^{-1/p}.
\end{equation*}
Therefore, 
\begin{equation}\label{eq77}
\|M(\chi _{B})^{1/s}\|_{\mathcal M^{p}(\varphi,w)'} \lesssim \sum_{k=1}^\infty 2^{k\left (\frac n{s'}+\frac{\lambda_1}p\right)} |B|^{1+\frac{\lambda_1}{np}}w(2^kB)^{\left (\frac {\lambda_2}n-1\right)\frac 1p}.
\end{equation}
(This estimate is valid for $p=1$ too, using  for $\|2^k\chi _{B}\|_{\mathcal M^{1}(\varphi,w)'}$ the bound obtained above.) The last exponent in \eqref{eq77} is negative, hence using that $w\in RD_\delta$ we have 
\begin{equation*}
w(2^kB)^{\left (\frac {\lambda_2}n-1\right)\frac 1p}\lesssim w(B)^{\left (\frac {\lambda_2}n-1\right)\frac 1p} 2^{-k\delta(n-\lambda_2)\frac 1p}.
\end{equation*}
Inserting this into \eqref{eq77} we get a convergent series by choosing $s$ close enough to $1$ such that the exponent of $2^k$ is negative and this is possible if $\delta>\lambda_1/(n-\lambda_2)$ as required. With this choice of $s$ we obtain 
\begin{equation}\label{eq78}
\|M(\chi _{B})^{1/s}\|_{\mathcal M^{p}(\varphi,w)'} \lesssim |B|^{1+\frac{\lambda_1}{np}}w(B)^{\left (\frac {\lambda_2}n-1\right)\frac 1p},
\end{equation}
which together with \eqref{norma1} implies \eqref{extrasuf}.

When $\lambda_2=n$ calculations are simpler (no reverse doubling condition is needed for $w$). In \eqref{eq77} the last term is missing and the series is convergent by choosing $s$ such that $n/s'+\lambda_1/p<0$, and this is always possible because $\lambda_1<0$. Then we have \eqref{eq78} without the last factor and this is what we need.

In the case $p>1$ we know that there exists $q\in (1,p)$ such that $w\in A_q$. Then  it satisfies \eqref{extrasuf} for such value of $q$. 

For $w\in A_p\cap RD_{\delta}$, $0<\lambda_2 \le n$, and $0>\lambda_1>-\delta \lambda_2$, the proof also works. The estimate \eqref{norma1} holds as was indicated in Remark \ref{rm73}, and the computations for $\|M(\chi _{B})^{1/s}\|_{\mathcal M^{p}(\varphi,w)'}$ are also valid up to \eqref{eq77}. Now the exponent of $2^k$ is clearly negative for $s$ close enough to $1$. 
\end{proof}

%\begin{remark}

%\end{remark}

In \cite{DR19,DR20} the result in the previous proposition was improved to weights of the form $v(x)=|x|^\alpha w(x)$ where $w\in A_p$ and $\alpha$ is a suitable positive number. We check directly that those weights are also included in the Muckenhoupt-Morrey condition. Notice that for some $q$ depending on $\alpha$, $v$ is in $A_q$, because $|x|^\alpha$ is a negative power of an $A_1$-weight.

The range of values of $\alpha$ is defined as follows. Assume that $w\in RD_\theta(0)$ in the sense of Section \ref{bi}, that is, 
\begin{equation}\label{RDtheta}
 \frac{w(B(0, r))}{ w(B(0, R)) }\lesssim \left(\frac rR\right)^{n\theta}\qquad \text{for all } 0<r<R.
\end{equation}
Assume that $0\le \lambda_2\le n$ and $\lambda_1 + \lambda_2 \theta> 0$. Then $\alpha$ satisfies 
\begin{equation} \label{bertiz5}
	0\leq\alpha<
	\begin{cases}
	  \frac{n}{n-\lambda_2 }\left(\lambda_1+ \lambda_2 \theta\right), & \quad\lambda_2<n,\\
	  \infty, & \quad\lambda_2= n.
  \end{cases}
\end{equation}

The following lemma is in \cite[Lemma 2.9]{DR20} (with less restrictions on $\lambda_2$).

\begin{lemma}\label{key2}
Let $w\in A_p$. (a) Let $B$ be a ball such that $r_B\le |c_B|/4$ and $f$ nonnegative. Then for arbitrary $\alpha$ it holds 
\begin{equation*}
\left(\frac 1{|B|}\int_B f\right)^p\le C r_B^{\lambda_1} v(B)^{\frac{\lambda_2}n-1}\|f\|_{\mathcal M^{p}(\varphi,v)}^p.
\end{equation*}
(b) Let $B$ centered at $0$ and $f$ nonnegative. Then for $0\le \lambda_2\le n$ and $\alpha$ as in \eqref{bertiz5} it holds
\begin{equation*}
\left(\frac 1{|B|}\int_B f\right)^p\le C r_B^{\lambda_1+\alpha \left(\frac{\lambda_2}n-1\right)}w(B)^{\frac{\lambda_2}n-1}\|f\|_{\mathcal M^{p}(\varphi,v)}^p.
\end{equation*}
\end{lemma}

These estimates can be written as estimates for $\|\chi_B\|_{\mathcal M^{p}(\varphi,v)'}$ by taking the supremum among those $f$ such that $\|f\|_{\mathcal M^{p}(\varphi,v)}^p\le 1$. More precisely, for $B$ such that $r_B\le |c_B|/4$ we have
\begin{equation}\label{smallball}
\|\chi_B\|_{\mathcal M^{p}(\varphi,v)'}\lesssim r_B^{n+\frac{\lambda_1}{p}} v(B)^{\frac{\lambda_2-n}{np}},
\end{equation}
and for $B$ centered at the origin we have
\begin{equation}\label{bigball}
\|\chi_B\|_{\mathcal M^{p}(\varphi,v)'}\lesssim r_B^{n+\frac{\lambda_1}{p}+ \frac{\alpha(\lambda_2-n)}{np}}w(B)^{\frac{\lambda_2-n}{np}}.
\end{equation}

\begin{proposition}\label{apextended}
Let $1\le p<\infty$, $0\le \lambda_1, \lambda_2 <n$ and $0<\lambda_1 + \lambda_2 < n$. Let $w\in A_p$ satisfying \eqref{rh2}, and let $\alpha$ be as in \eqref{bertiz5}. If $v(x)=|x|^\alpha w(x)$, then $v\in A(\mathcal M^{p}(\varphi))$. 

Even more if $w\in RD_{\delta}$ with $\delta>\lambda_1/(n-\lambda_2)$, then $v$ satisfies the stronger condition \eqref{extrasuf} for $q=p$, and also for some $1<q<p$ if $p>1$. 

The same result holds for $w\in A_p\cap RD_{\delta}$, $0<\lambda_2 \le n$, $0>\lambda_1>-\delta \lambda_2$, and $\alpha$ as in \eqref{bertiz5}. 
\end{proposition} 

\begin{proof}
(a) Let $B$ be such that $|c_B|>4r_B$. Then for $y\in B$ we have $|y|^\alpha\sim |c_B|^\alpha$, from which it follows that $v(B)\sim |c_B|^\alpha w(B)$. Proceeding as in the proof of Proposition \ref{apsuf} we have instead of \eqref{norma1}
\begin{equation*}
\|\chi _{B}\|_{\mathcal M^{p}(\varphi,v)} \le \left( \frac {[|c_B|^\alpha w(B)]^{1-\frac{\lambda_2}n}}{|B|^{\frac{\lambda_1}n} }\right)^{\frac 1p}.
\end{equation*}
(This is obtained as in Remark \ref{rm73} observing that for $x\in 2B$ we have $|x|\sim |c_B|$.) Using this estimate together with \eqref{smallball} we see that $v$ satisfies the $A(\mathcal M^{p}(\varphi))$ condition. 

(b) Let $B$ be centered at the origin. The bound for $\|\chi _{B}\|_{\mathcal M^{p}(\varphi,v)}$ is now
\begin{equation*}
\|\chi _{B}\|_{\mathcal M^{p}(\varphi,v)} \le \left( \frac {[r_B^\alpha w(B)]^{1-\frac{\lambda_2}n}}{|B|^{\frac{\lambda_1}n} }\right)^{\frac 1p}.
\end{equation*}
For $\|\chi _{B}\|_{\mathcal M^{p}(\varphi,v)'}$ we have \eqref{bigball}. It follows that $v\in A(\mathcal M^{p}(\varphi))$.

(c) For the stronger condition \eqref{extrasuf} when $w\in RD_{\delta}$, we proceed as in the proof of Proposition \ref{apsuf}. The bounds of $\|\chi _{B}\|_{\mathcal M^{p}(\varphi,v)}$ are in parts (a) and (b). On the other hand, we use the pointwise bound \eqref{pb74} and take the bounds  \eqref{smallball} and \eqref{bigball} for the norm of $\chi_{2^kB}$ in $\mathcal M^{p}(\varphi,v)'$. If $B$ is centered at the origin, all the balls $2^kB$ are also centered at the origin and we only need \eqref{bigball}. If $B$ is such that $|c_B|>4r_B$, we use  \eqref{smallball} for those values of $k$ such that $|c_B|>4\cdot 2^k r_B$, and \eqref{bigball} for the other values of $k$, after replacing $2^kB$ with $\widetilde{2^kB}$ (using the notation introduced in Section \ref{zortzi}). The details are left to the reader.
\end{proof}

\subsection{The case of weighted local Morrey spaces}

Lemma \ref{emblema} is clearly true for $\mathcal {LM}^{p}(\varphi,v)$, which contains $\mathcal M^{p}(\varphi,v)$.

Part (a) of Proposition \ref{rhprop} holds for $A(\mathcal {LM}^{p}(\varphi))$ weights if $B_1$ and $B_2$ are centered at the origin (that is, $w\in RD_{\frac {\lambda_1}{n-\lambda_2}}(0)$), and part (b) holds with $A_{\frac{np+\lambda_1}{n-\lambda_2},0}$. (The class $A_{p,0}$ is like the usual $A_p$ class, but the condition is restricted to balls centered at the origin.)

The first part of Proposition \ref{apsuf} for the local weights is the following: if $w\in A_{p,0}\cap RD_{\frac {\lambda_1}{n-\lambda_2}}(0)$, then $w\in A_0(\mathcal{LM}^p(\varphi))$ (we have \eqref{norma1} due to the embedding of $\mathcal{M}^p(\varphi,w)$ into $\mathcal{LM}^p(\varphi,w)$). In the proof of the second part of the proposition we cannot use now that $w\in A_{p,0}$ implies $w\in A_{q,0}$ for some $q<p$ (see \cite{DMO13}). But when we want to apply Theorem \ref{teo1Kloc}, we also require $w\in A_{p, \textrm{loc}}$, and $A_{p,0}\cap  A_{p, \textrm{loc}}=A_p$. Therefore, Proposition \ref{apsuf} is valid for weighted local Morrey spaces instead of weighted global Morrey spaces. The same thing applies to the extension to weights of the form $|x|^\alpha w(x)$ of Proposition \ref{apextended}. Here we only need part (b) of the proof of the proposition and \eqref{bigball}.

\subsection{Applications to operators}
Combining the extrapolation results of Section \ref{sei} with the particular results in this section, we obtain immediately for the corresponding weighted global and local Morrey spaces the boundedness of the operators which are bounded on $L^p(w)$ with $w\in A_p$. In particular, we can mention Calder\'on-Zygmund operators and their associated maximal operators, square functions (of Littlewood-Paley type and others), rough operators with hmogeneous kernel $\Omega(x/|x|)|x|^{-n}$ and $\Omega\in L^\infty$ with vanishing integral, Bochner-Riesz operators at the critical index, commutators, vector-valued extensions, etc. Weak-type estimates for $p=1$ are obtained too (except for commutators). In the setting of weighted local Morrey spaces several of these results are new.

\subsection{Power weights}

Let us consider the particular case $w(x)=|x|^\beta$ for $\beta>-n$. These weights were thoroughly studied in \cite{DR20} for $\varphi$ as in \eqref{embed} and a larger range of values of $\lambda_1$ and  $\lambda_2$. Let us apply the results in this section for glocal and local Morrey spaces.

First notice that $|x|^\beta\in RD_\delta$ for $\delta=\min (1, 1+\beta/n)$. In particular, the exponent is less than $1$ only for $\beta<0$. If we restrict to balls centered at the origin, $|x|^\beta\in RD_{1+\beta/n}(0)$. This is the value of $\theta$ appearing in \eqref{RDtheta} for power weights. 

Let $0\le \lambda_1,\lambda_2<n$ and $0<\lambda_1+\lambda_2<n$. Then $w\in A_p\cap RD_{\frac {\lambda_1}{n-\lambda_2}}$ if and only if
\begin{equation}\label{powersmall}
-\frac{n(n-\lambda_1-\lambda_2)}{n-\lambda_2}\le \beta <n(p-1). 
\end{equation}
The left bound is imposed by the reverse doubling condition and the right one by the $A_p$ condition. (The same bounds are imposed if we consider only balls centered at the origin.) Those weights are in $A(\mathcal M^{p}(\varphi))$ according to Proposition \ref{apsuf}. But we can improve the right bound using Proposition \ref{apextended}. Indeed, if $w(x)=|x|^\gamma\in A_p$ we have $v(x)=|x|^{\alpha+\gamma}\in A(\mathcal M^{p}(\varphi))$ with $\alpha$ given by \eqref{bertiz5}: 
\begin{equation*}
0<\alpha<\frac{n}{n-\lambda_2}\left[\lambda_1+\lambda_2\left(1+\frac{\gamma}{n}\right)\right].
\end{equation*}
Since $\gamma$ can be arbitrarily close to $n(p-1)$, the range in \eqref{powersmall} can be improved to 
\begin{equation}\label{powerbig}
-\frac{n} {n-\lambda_2}(n-\lambda_1-\lambda_2)\le \beta <\frac{n} {n-\lambda_2}[n(p-1)+\lambda_1+\lambda_2]. 
\end{equation}
In the case of $0<\lambda_2<n$ and negative $\lambda_1$, the left bound is $-n$. In the case $\lambda_2=n$, the range is $-n<\beta<\infty$. Under the conditions of Proposition \ref{rhprop} the range in \eqref{powerbig} is sharp for $w\in A(\mathcal M^{p}(\varphi))$, because part (b) imposes the right-hand side. It is also sharp for $0<\lambda_2<n$ and negative $\lambda_1$ (see  \cite[Proposition 3.2]{DR20}).

For the stronger condition \eqref{extrasuf} the left point $\beta=-\frac{n(n-\lambda_1-\lambda_2)}{n-\lambda_2}$ is lost because we need a slightly better reverse doubling condition. 

For the relevant cases appearing in the literature these results particularize to the following. Let $w(x)=|x|^\beta$ and
\begin{itemize}
\item $\varphi(B)=r_B^\lambda$, $0<\lambda<n$ (Samko-type): $-n+\lambda\le \beta<n(p-1)+\lambda$;
\item $\varphi(B)=w(B)^{\lambda/n}$, $0<\lambda<n$ (Komori-Shirai-type): $-n<\beta< \frac{n}{n-\lambda}[n(p-1)+\lambda]$;
\item $\varphi(B)=r_B^{-\lambda}w(B)$, $0<\lambda<\min(n,n+\beta)$ (Poelhuis-Torchinsky-type): $-n+\lambda<\beta<\infty$. This case corresponds to $\lambda_1=-\lambda$ and $\lambda_2=n$, and the left bound comes from $\lambda_1+\delta \lambda_2>0$, imposed to ensure that $\varphi$ is reverse doubling. However, it was proved in \cite[Proposition 2.10(c)]{DR20} that if $\lambda_1+\delta \lambda_2<0$ the Morrey space is $\{0\}$.
\end{itemize}

The operator $M$ is bounded on $\mathcal{M}^p(\varphi)$ and on $\mathcal{LM}^p(\varphi)$ if $1<p<\infty$ and $\beta$ is in the given range. (In \cite{DR20} the left endpoint needed a specific proof because a reverse H\"older inequality was used instead of a reverse doubling one.) For $p=1$ it is bounded from $\mathcal{M}^1(\varphi)$ to $W\mathcal{M}^1(\varphi)$ and from $\mathcal{LM}^1(\varphi)$ to $W\mathcal{LM}^1(\varphi)$. For $M_0$ the estimate is of strong type also for $p=1$. 

For operators $T$ for which the pairs $(|f|,|Tf|)$ satisfy \eqref{hypextra}, the same results hold for $1<p<\infty$ except for the left endpoint, when it is reached in the range of $\beta$. The weak estimates for $p=1$ also hold if \eqref{hypextra} is satisfied for $p=1$ with a weak-type estimate and $A_1$ weights.

%%%%%%%%%%%%%%%%%%%%%%%%%%%%%%%%%%%%%%%%%%%%%%%%%%%%%%%%%%%%%%%

\section*{Acknowledgements}

The first author is supported by the grants MTM2017-82160-C2-2-P of the Ministerio de Econom\'{\i}a y Competitividad (Spain) and FEDER, and IT1247-19 of the Basque Gouvernment. 
\par\smallskip

%%%%%%%%%%%%%%%%%%%%%%%%%%%%

\end{document}